\documentclass[10pt]{amsart}
\usepackage[colorlinks]{hyperref}
\usepackage{color,graphicx,shortvrb}
\usepackage[latin 1]{inputenc}
\usepackage{mathtools} 
\usepackage[active]{srcltx} 

\usepackage{enumerate}

\newtheorem{theorem}{Theorem}[section]

\newtheorem{definition}[theorem]{Definition}
\newtheorem{lemma}[theorem]{Lemma}

\newtheorem{proposition}[theorem]{Proposition}
\newtheorem{remark}[theorem]{Remark}

\def\RR{{\mathbb{R}}}

\def\11{\textbf{$1$}}

\def\ir{\displaystyle{\int_{\RR^N}}}
\def\dys{\displaystyle}
\def\rife{\eqref}
\def\inn{\mbox{ in }\quad}

\DeclareGraphicsExtensions{.jpg,.pdf,.png,.eps}

\DeclareMathOperator{\supp}{supp}
\DeclareMathOperator{\dis}{dist}

\numberwithin{equation}{section}


\title{Parabolic equations with natural growth approximated by nonlocal equations}
\author[T. Leonori, A. Molino Salas,  S. Segura de Le\'on]
{T. Leonori, A. Molino Salas,  S. Segura de Le\'on}
\thanks{}

\address{Tommaso Leonori
\hfill \break\indent Departamento de An\'alisis Matem\'atico, Universidad de Granada,
\hfill\break\indent Avenida Fuentenueva S/N,18071 Granada, Spain} \email{{\tt leonori@ugr.es}}

\address{Alexis Molino Salas
\hfill \break\indent Departamento de An\'alisis Matem\'atico, Universidad de Granada,
\hfill\break\indent Avenida Fuentenueva S/N,18071 Granada, Spain} \email{{\tt amolino@ugr.es }}

\address{ Sergio Segura de Le\'on.
\hfill\break\indent Departament d'An\`alisi Matem\`atica, Universitat de Val\`encia,
\hfill\break\indent Dr. Moliner 50,
46100 Burjassot, Valencia, Spain} \email{{\tt sergio.segura@uv.es.}}

\begin{document}

\begin{abstract}
In this paper we study several aspects related with solutions of nonlocal problems whose prototype is
$$
\begin{cases}
u_t  =\displaystyle \int_{\RR^N} J(x-y) \big( u(y,t) -u(x,t) \big) \mathcal G\big( u(y,t) -u(x,t)   \big) dy     \qquad & \mbox{ in } \, \Omega \times (0,T)\,,\\
u(x,0)=u_0 (x) & \mbox{ in }   \Omega, \
\end{cases}
$$
 where we take, as the most important instance, $\mathcal G (s) \sim 1+ \frac{\mu}{2} \frac{s}{1+\mu^2 s^2 }$ with $\mu\in \RR$ as well as $u_0 \in L^1 (\Omega)$, $J$ is a smooth symmetric function with compact support  and $\Omega$ is either a bounded smooth subset of $\RR^N$, with nonlocal Dirichlet boundary condition, or $\RR^N$ itself.

 \medskip

 The results deal with existence, uniqueness, comparison principle and asymptotic behavior. Moreover we prove that if the kernel rescales in a suitable way, the unique solution of the above problem converges to a solution of the deterministic Kardar-Parisi-Zhang equation.

\end{abstract}

\maketitle

\section{Introduction}

This work is concerned with the study the existence, uniqueness, comparison principle and asymptotic behavior for the following nonlinear parabolic equation with nonlocal diffusion,
\begin{equation}\label{nonlocal_gen}
\begin{cases}
u_t(x,t)  =\displaystyle \int_{\RR^N} J(x-y) \big( u(y,t) -u(x,t) \big) \mathcal G\big( u(y,t) -u(x,t)   \big) dy     \qquad & \mbox{ in } \, \Omega \times (0,T)\,,\\
u(x,0)=u_0 (x) & \mbox{ in }   \Omega, \
\end{cases}
\end{equation}
for an appropriate functions $J$ and $\mathcal{G}$ (see below \eqref{(H)} and \eqref{Gi}), and its relationship with the local problem
\begin{equation}\label{KPZ}
\begin{cases}
u_t -\Delta  u=\mu |\nabla   u|^2  \qquad &\hbox{ in } \Omega \times (0,T)\,,\\[1.5 ex]
u (x,0)= u_0 (x) &\hbox{ in }\Omega\,, 
\end{cases}
\end{equation}
where
\begin{enumerate}
  \item $\Omega$ is either $\RR^N$ itself (Cauchy problem) or a bounded smooth subset of $\RR^N$ adding the boundary condition $u(x,t)=h(x,t)$ on $\partial \Omega \times (0,T)$ for $h$ sufficiently smooth (Dirichlet problem);
  \item $T>0$ (possibly infinite) and   $\mu\in  \RR $;
  \item $u_0$ is a smooth enough datum.
\end{enumerate}

\subsection{Local problem}

The equation $u_t -\Delta  u=\mu |\nabla   u|^2$ , at least for $\mu>0$, is known in the literature as the deterministic Kardar-Parisi-Zhang (KPZ) equation. It was proposed in \cite{K} in the physical theory of growth and roughening of surfaces. Further developments on physical applications of the KPZ equation can be found in \cite{BS} (for a survey on more recent aspects we refer to \cite{WERDL}).

\smallskip

The Kardar--Parisi--Zhang equation has given rise to a rich mathematical theory which has had a spectacular recent progress (see \cite{C,H}). From the point of view of Partial Differential Equations, equations having a gradient term with the so-called natural growth have been largely studied in the last decades by many mathematicians: in addition to the classical reference \cite{LSU} let us just mention the pioneer paper by Aronson and Serrin \cite{AS} and also the result due to Boccardo, Murat and Puel \cite{BMP}.

\subsection{Nonlocal problem}

Nonlocal evolution equations have been extensively studied to model diffusion processes. The prototype example in this framework is the following one
\begin{equation}\label{original}
u_t(x,t)=\displaystyle \ir K(x,y)(u(y,t)-u(x,t))dy,
\end{equation}
where the kernel
$K:\RR^N \times \RR^N \to \RR$ is a nonnegative smooth function (not necessarily symmetric) satisfying $\displaystyle \ir K(x,y)dx=1$ for any $y\in \RR^N$
(or variations of it, see for instance \cite{AMRT}).  If $u(y,t)$ is thought of as a density at location  $y$ at time $t$ and $K(x,y)$ as the probability distribution of jumping from place $y$ to place $x$, then the rate at which individuals from any other location go to the place $x$ is given by $ \ir K(x,y)u(y,t)dy$. On the other hand, the rate at which individuals leave the location $x$ to travel to all other places is $-\ir K(y,x)u(x,t)dy=-u(x,t)$. In the absence of external sources this implies that the density must satisfy equation \eqref{original}.

\medskip

We are especially interested in symmetric kernels (we denote them by $J$) that have compact support; it means that the individuals can jump from a place to other, but they cannot go \lq\lq too far away\rq\rq. On the contrary, for instance, nonlocal operators that allow \lq\lq long jumps\rq\rq  correspond to a different choice of kernels. It is the case of the fractional laplacian that involves a  kernel that is singular and that does not have compact support  (see, for instance \cite{V} for a  survey on this latter class of processes).

\medskip

In particular, we consider  $J: \RR^N \to \RR$ as a nonnegative radial symmetric function such that
\begin{equation*}
J \in \mathcal{C}_c(\RR^n), \,\quad \,\,\int_{\RR^N}J(z)\,dz=1\,\,\quad \, \textrm{and}\,\,\,\quad  \ir J(z)z_N^2dz <\infty, \qquad z=(z_1, \dots, z_N).
\end{equation*}

With this choice of the kernel, equation \rife{original} changes into a diffusion equation  of convolution type, namely
\begin{equation}\label{classical}
u_t(x,t)=(J\ast u-u)(x,t)=\ir J(x-y)u(y,t)dy-u(x,t), \qquad \mbox{ in } \Omega \times (0, T)
\end{equation}
(see for instance  \cite{Bates,Chasse,CERW}).

\subsection{Background}

  One of the most important features of nonlocal equations is that can be rescaled to approximate local ones.

In \cite{Elgueta} (see also \cite{MR} for the same type of result in a  more general case) it has been proved   that, under an appropriate rescaling kernel, solutions of \eqref{classical} converge uniformly to solutions of heat equation. To be more specific, solutions of
\begin{equation}\label{heat}
u_t^\varepsilon(x,t)=\displaystyle \frac{C}{\varepsilon ^2}\left[\ir J_\varepsilon (x-y)u(y,t)dy-u(x,t)\right]  \qquad \mbox{ in } \Omega \times (0,T)
\end{equation}
converge uniformly to solutions of
$$v_t=\Delta v \qquad \mbox{ in } \Omega \times (0,T) \,, $$
where $\displaystyle C^{-1}=\frac{1}{2}\int_{\RR^N}J(z) z^2_N\, dz$ and $J_\varepsilon (s)=\displaystyle\frac{1}{\varepsilon^N}J(\frac{s}{\varepsilon})$.

\medskip

Let us mention that results in this direction, with the presence of a gradient term of convection type can be found, for instance, in \cite{IR}: in such a case the equation is the sum of two terms, one corresponding to the  diffusion one, the other to the convection term.


In general, we   consider nonlocal problems of the type
\begin{equation}\label{general0}
u_t (x,t)=\displaystyle \ir J(x-y)\big(u(y,t)-u(x,t)\big)\,\mathcal{G}\big(u(y,t)-u(x,t)\big)\,dy,
\end{equation}
where $\mathcal{G}:\RR\to\RR$  is a suitable continuous function. For instance, if $\mathcal{G}\equiv 1$, then we recover problem \eqref{classical}. Let us mention the case $\mathcal{G}(s)=|s|^{p-2}$, with $p\geq 2$ has been treated in \cite{AMRT} where it is proved that solutions to the rescaled nonlocal problem converge to solutions of the Dirichilet problem for the $p$--Laplacian evolution equation.

\medskip

On the contrary, the kind of kernels  $\mathcal{G}$ we consider does  not have the same structure of the previous ones, since they are bounded and  do not satisfy any symmetry assumptions (neither odd nor even).

\medskip

With this background, it is not surprising that problem \eqref{KPZ} can be approximated by nonlocal equations. The question is to identify what kind of nonlocal equation approximates, under rescaling, problem \eqref{KPZ}. At first glance, one could think that a good approximation for \eqref{KPZ} might be a nonlocal equation such as
\begin{equation*}
  u_t (x,t)=\displaystyle \ir J(x-y)\big(u(y,t)-u(x,t)\big)\,dy+\mu\ir J(x-y)|u(y,t)-u(x,t)|^2\,dy,
\end{equation*}
that is, taking $\mathcal G(s)=1+\mu s$ in \eqref{general0}. We explicitly point out that this is an unbounded function that satisfies $\mathcal G(0)=1$ and $\mathcal G^\prime(0)=\mu$ (compare with condition \eqref{Gi} below). Anyway, for
our approach the lack of boundedness of  $\mathcal{G}$ leads to an obstacle for proving the existence of a solution to \rife{general0} via a fixed point argument. By the other hand, we recall that one of the main tools to deal with problem \eqref{KPZ} is the so--called Hopf--Cole change of unknown which is defined by $w(x,t)=e^{\mu u(x,t)}$. This transforms every classical solution to  \eqref{KPZ} into a classical solution to problem
\begin{equation*}
\left\{
\begin{array}{ll}
w_t(x,t)=\Delta w(x,t)
& \hbox{ in } \Omega \times (0,T)\,,\\[1.5 ex]
w(x,0)=e^{\mu   u_0(x)}
&\hbox{ in }\Omega\,,
\end{array}\right.
\end{equation*}
for a smooth enough datum $u_0$. However, the same kind of difficulty are found if one try to reproduce the  Hopf--Cole transformation and try to  approximate the solution of \eqref{KPZ}   by something of the form
\begin{equation*}
  u_t (x,t)= \ir J(x-y)\Big(e^{\mu u(y,t)}-e^{\mu u(x,t)}\Big)\,dy\,.
\end{equation*}

\subsection{Main results} To conclude this introduction we want to state the most relevant results of our work. In order to not enter in technicalities, let us fix a family of kernels $  \mathcal G_\mu $ that are the easiest (not trivial) example we can consider: for $\mu \in \RR$ let
\begin{equation*}
  \mathcal G_\mu (s) =1+\frac{\mu s}{2(1+\mu^2 s^2)}\,,\qquad s\in\RR\,,\quad  \mu \in \RR,
\end{equation*}
and the corresponding family of nonlocal Dirichlet problems
\begin{equation} \label{model-int}
\left\{
\begin{array}{ll}
u_t(x,t)= \displaystyle
\int_{\RR^N}  J(x-y) \big(u(y,t)-u(x,t) \big)\,\mathcal{G}_\mu\big(u(y,t)-u(x,t)\big)\,dy
&\hbox{ in } \  \Omega \times (0,T) \,,\\[1.5 ex]
u (x,0)=u_0(x)&   \inn \Omega\,,\\[1.5 ex]
u(x,t)=h(x,t) &\inn (\RR^N \setminus \Omega) \times (0,T)\,.
\end{array}
\right.
\end{equation}
with $\Omega$ a bounded domain and  $u_0$ and $h$ smooth enough.

\medskip

After have proved the existence, uniqueness  (see Theorem  \ref{existence}) and a Comparison Principle (see Theorem \ref{C-P}) for solutions of \rife{model-int}, we face the problem of {\it rescaled kernels}.

\medskip

The result we prove, in this model case, reads like this.

%

%

\medskip

 {\it
 Let  $u$ be the unique smooth solution to \eqref{KPZ}, with suitable initial data $u_0$ and boundary condition $u(x,t)=h(x,t)$ on $\partial \Omega \times (0,T)$. Then there exists a family of functions   $\{u^\varepsilon\}$, $\varepsilon>0$, such that $u^\varepsilon$  solves the approximating nonlocal problem
 $$
\left\{
\begin{array}{l}
u_t^\varepsilon(x,t)=
\displaystyle \frac{C }{\varepsilon^2} 
\displaystyle
\int_{\Omega_{J_\varepsilon}}  J_\varepsilon(x-y) \bigg[\big(u^\varepsilon(y,t)-u^\varepsilon(x,t) \big)
+ \frac{\mu}{2}\frac{  \big(u^\varepsilon(y,t)-u^\varepsilon(x,t) \big)^2 }{1+\mu^2\big(u^\varepsilon(y,t)-u^\varepsilon(x,t) \big)^2 }
\bigg]
 dy
\quad   \mbox{ in } \  \Omega\times (0,T)  ,\\[2.0 ex]
u^\varepsilon (x,0)=u_0(x)
\qquad \qquad   \qquad  \quad   \qquad \quad \qquad \qquad  \qquad \qquad  \qquad \qquad  \qquad \qquad  \qquad \qquad
\qquad   \inn \Omega,\\[1.5 ex]
u^\varepsilon(x,t)=h(x,t)
\qquad \qquad  \qquad  \qquad    \qquad \qquad  \qquad \qquad  \qquad \qquad  \qquad \qquad  \qquad
\inn (\Omega_{J_\varepsilon} \negthickspace \setminus \negthickspace \Omega) \times (0,T),
\end{array}
\right.
$$
with $C$ a suitable constant, $\Omega_{J_\varepsilon}= \Omega+\supp J_\varepsilon$ and the family $\{u^\varepsilon\}$ satisfies
$$
\lim_{\varepsilon \to 0} \quad \sup_{t\in [0,T]} \Big\|u^\varepsilon(x,t)-u(x,t)\Big\|_{L^{\infty} (\Omega)} =  0 \,.
$$
}

 \medskip

The same kind of results (i.e. existence, uniqueness and convergence for a  suitable rescaled kernel to a solution of a local problem) are also proved for the corresponding Cauchy problem associated (i.e., $\Omega=\RR^N$).

 \medskip

 In addition, we deal with the asymptotic behavior of the solutions of problem \eqref{nonlocal_gen}. Concretely, we have two kind of results: if $\Omega $ is a bounded domain of $\RR^N$, we prove that the solutions of \rife{model-int} converge uniformly  to the stationary one. On the other hand,  if $\Omega= \RR^N$, we prove that the $L^2$-norm of the solution has a suitable decay in time, depending on the nature (absorption or reaction) of the kernel (see for more details Theorems \ref{asymp Cauchy} and \ref{L^2 mu>0}, respectively).

\subsection*{Plan of the paper}

Section 2 is devoted to show the precise statements of the main results. Preliminaries are contained in Section 3. Section 4 deals with the Dirichlet problem in a bounded domain, while the results concerning the Cauchy problem can be found in Section 5.

\section{Statement of the   results}

This section is devoted to the statement of the main results we prove in the present paper.

Let us  consider the following equation:
\begin{equation}\label{general}
u_t (x,t)=\displaystyle \ir J(x-y)u(y;x,t)\,\mathcal{G}(x,u(y;x,t))\,dy,
\end{equation}
where  $J: \RR^N \to \RR$ is a nonnegative radial symmetric function such that
\begin{equation*}\label{(H)}
J \in \mathcal{C}_c(\RR^n), \qquad \int_{\RR^N}J(z)\,dz=1 \qquad \textrm{and}\qquad  C(J) := \ir J(z)z_N^2dz<\infty, \quad z=(z_1, z_2, \dots, z_N)
\tag{$J$}
\end{equation*}
and where, here and throughout the paper, we denote $u(y;x,t):=u(y,t)-u(x,t)$.

As far as the function $\mathcal{G}$ is concerned, we assume that  $\mathcal{G}:\RR^N\times\RR\to\RR$   is  a nonnegative   Carath\'eodory  function   (namely, $\mathcal G(\cdot, s)$ is measurable for every $s\in\RR$ and $\mathcal G(x,\cdot)$ is continuous for almost every $x\in\RR^N$) satisfying
\begin{equation}\label{Gi}\tag{$\mathcal{G}$}
\exists\  \alpha_2\geq \alpha_1 >0\,: \qquad  \alpha_1\leq\displaystyle \frac{ \mathcal{G}(x,s)s-\mathcal{G}(x,\sigma)\sigma}{s-\sigma}\leq \alpha_2, \hspace{1,1cm} \forall s, \sigma\in\RR\,\ s\ne \sigma\,, \hbox{ and for a.e.  }x\in \RR^N.
\end{equation}

Let us first point out that the above  condition  implies that $\mathcal{G}$ is a positive bounded function, since  taking  $\sigma=0$ in \eqref{Gi}, we   get
\begin{equation*}
0<\alpha_1\leq\mathcal{G}(x,s)\leq \alpha_2, \qquad \hbox{for any  }s\in\RR \hbox{  and for  a.e. }x\in \RR^N.
\end{equation*}
Moreove observe that   the above condition relies to be a sort of uniform ellipticity for the operator, while \rife{Gi} corresponds to a strong monotonicity.

\medskip

Further remarks about the  condition on $\mathcal{G}$ are addressed to Section 3.

\medskip

Anyway, let us stress again that, in contrast with all the known results about nonlocal equation of the above type, in our case we do not require any symmetry (neither odd nor even) assumption to $\mathcal{G}$.

\bigskip

The prototype  of   $\mathcal{G}$ we have in mind (we will come back on this example later)  is the following one:
$$
\mathcal{G}_\mu(x,s)=1+\frac{\mu(x) \, s}{2(1+\mu(x)^2s^2)},\quad  x\in \Omega, \quad s\in \RR\,,
$$
where $\mu:\Omega \to \RR$ stands for a measurable function.

 \subsection{Dirichlet problem}

The first kind of results we want to prove deals with the existence and uniqueness of solutions of a nonlocal Dirichlet boundary value problem. More precisely, consider the following problem in a   bounded domain $\Omega \subset \RR^N$, $N\geq 1$.
\begin{equation*}
\left\{
\begin{array}{ll}
u_t (x,t)=\displaystyle \ir J(x-y)u(y;x,t)\,\mathcal{G}(x,u(y;x,t))\,dy,
& \mbox{ in } \Omega \times (0,T)
\\[1.5 ex]
u(x,t)=h(x,t), & \mbox{ in } (\RR^N \setminus \Omega) \times (0,T),
\\[1.5 ex]
u(x,0)=u_0(x),
& \mbox{ in } \Omega,
\end{array}\right.
\end{equation*}
with  $h \in L^1\left( (\RR^N\setminus \Omega)\times (0,\infty)\right)$ and $u_0 \in L^1(\Omega)$.

\medskip

  Let us first observe  that the integral expression vanishes outside of $\Omega_J=\Omega + \supp(J)$. In this way, $h$ is only needed to be fixed, in fact,  in $\Omega_J \setminus \Omega$ and we can rewrite the above problem as
\begin{equation*}\label{P}
\left\{
\begin{array}{ll}
u_t (x,t)=\displaystyle \int_{\Omega_J} J(x-y)u(y;x,t)\,\mathcal{G}(x,u(y;x,t))\,dy,
& \mbox{ in } \Omega \times (0,T),
\\[1.8 ex]
u(x,t)=h(x,t), &  \mbox{ in }  (\Omega_J\setminus\Omega) \times (0,T) ,
\\[1.5 ex]
u(x,0)=u_0(x),
& \mbox{ in } \Omega,
\end{array}\right.
\tag{$P$}
\end{equation*}
where $T>0$ may be finite or $+\infty$.

\medskip

Due to the aim of the paper, we give now two definitions of solution.
\begin{definition}
Assume that $J$ and $\mathcal{G}$ satisfy \eqref{(H)} and \eqref{Gi}, respectively. \\
For $h(x,t) \in    L^1((\Omega_J\setminus \Omega)\times (0,T))$ and $u_0 (x) \in L^1 (\Omega)$, we define a {\rm weak solution} of problem \eqref{P}   a function $u \in \mathcal{C}([0,T);L^1(\Omega))$ such that:
\begin{gather}
\label{cond1} u(x,t)=\int_0^t \int_{\Omega_J} J(x-y)u(y;x,\tau) \mathcal{G}(x,u(y;x,\tau))\,dy\,d\tau+u_0(x), \qquad  \mbox{ for a.e. x }   \in \Omega,\, t\in ( 0,T),\\
\nonumber  u(y,t)=h(y,t)\qquad\hbox{for a.e. } y \in \Omega_J\setminus \Omega\hbox{ and }  t\in ( 0,T)\\
\nonumber \lim_{t\to 0^+} \| u(x,t) - u_0 (x) \|_{L^1 (\Omega)} =0\,.
\end{gather}

Moreover, if $h(x,t) \in    \mathcal{C}((\Omega_J\setminus \overline{\Omega})\times (0,T))$ and $u_0 (x) \in \mathcal{C} (\overline{\Omega})$, we define a {\rm regular solution} of problem \eqref{P} as a function $u \in \mathcal{C}([0,\infty);\mathcal{C}(\overline{\Omega}))$ such that:
\begin{gather}
\nonumber u(x,t)=\int_0^t \int_{\Omega_J} J(x-y)u(y;x,\tau)\mathcal{G}(x,u(y;x,\tau))\,dy\,d\tau+u_0(x), \qquad  \mbox{ for any  x } \in \overline{\Omega},\,  t\in ( 0,T),\\
\nonumber u(y,t)=h(y,t)\qquad\hbox{for any } y \in \Omega_J\setminus \overline{\Omega} \hbox{ and }  t\in ( 0,T)\\
\nonumber \lim_{t\to 0^+} \| u(x,t) - u_0 (x) \|_{\mathcal{C}(\overline{\Omega})} =0\,.
\end{gather}
\end{definition}

Some more remarks about the meaning of weak and regular solutions are in order to be given.
\begin{remark} \
\begin{itemize}
\item[i)]Observe that, in addition to the different smoothness of the boundary condition and/or the initial datum, the main difference lies on the prescription of data on  $\partial\Omega$. Indeed, for weak solutions, $h$ is prescribed in $(\Omega_J\setminus \Omega)\times (0,T)$ and $u_0$ in $\Omega$, while for regular solutions, $h$ is prescribed in $(\Omega_J\setminus \overline\Omega)\times (0,T)$ and $u_0$ in $\overline\Omega$.

\item[ii)]  As already noticed in \cite{Chasse} (in a different context) the {\rm boundary conditions} cannot be meant in a classical way, i.e. it is not true that the solutions of problem \eqref{P} pointwise coincide with the prescribed boundary data $h(x,t)$. This is due to the fact that the value at any point    $(x,t) \in \partial \Omega \times (0,T)$ depends both on the values of $u$ inside $\overline \Omega \times [0,T]$ and on the boundary datum $h(x,t)$, since
$$
\begin{array}{c}
\displaystyle
u(x,t)=\int_0^{t} \int_{\Omega \cap {\text{supp} J}} J(x-y)u(y;x ,\tau)\ \mathcal{G}\big(x, u(y,\tau)- u(x,\tau) \big)\,dy\,d\tau
\\
\displaystyle
+ \int_0^{t} \int_{\Omega^c \cap {\text{supp} J}} J(x-y)\big(h(y,\tau) - u(x ,\tau) \big) \ \mathcal{G}\big(x, h(y,\tau)- u(x,\tau) \big) \,dy\,d\tau
+u_0(x)\,.
\end{array}$$
Consequently, in contrast with   the local case,   the equation is solved up to the boundary, depending,   near $\partial \Omega$,  also of the prescribed boundary condition.

\item[iii)] Let us stress that the {\rm regularity} required in the   definition of weak solutions is the less restrictive in order to give sense to the formulation and to the boundary and initial conditions.
Anyway from \eqref{cond1}  we deduce that   the time derivative $u_t (x,t) $ of $u$ also  belongs to $\mathcal{C}((0,\infty);L^1(\Omega))$.

Let us also point out that the weak solutions framework is the more natural one in order to prove the existence of  a solution. Indeed we only require an $L^1$ regularity  to prove the existence of a solution.

Finally we want to underline that the nonlocal operator involved in such equation does not have the regularizing effect that is typical of the Laplacian, but leave unchanged the regularity of the initial  and boundary data.

\end{itemize}
\end{remark}

Our existence result is the following.

\begin{theorem}{[Existence]}\label{existence} Consider problem \eqref{P} and suppose that \rife{(H)} and \rife{Gi} are in force. Then:
\begin{itemize}
\item[i)] For any $u_0 \in L^1(\Omega)$ and $h \in L^1((\Omega_J\setminus \Omega)\times (0,T))$ there exists a unique weak solution;
 \item[ii)] For any     $u_0 \in \mathcal C(\overline\Omega)$ and $h \in \mathcal C((\Omega_J\setminus \overline\Omega)\times [0,T))$   there exists a unique regular  solution and  moreover
its time derivative    belongs to $
  \mathcal C(\overline\Omega\times(0,T))
$.
\end{itemize}
\end{theorem}

Once we have deduced the existence of a solution, one important tool is to compare two solutions, or, more generally a sub and a  supersolution. Here we recall what we mean by those concepts in our setting.

\begin{definition} A function $u \in \mathcal C(\overline\Omega\times [0,T])$ is a regular subsolution to problem \eqref{P} if it satisfies
$u_t\in \mathcal  C(\overline\Omega\times (0,T))$ and
\begin{equation}\label{p}
\left\{
\begin{array}{ll}
u_t(x,t) \leq \displaystyle \int_{\Omega_J} J(x-y)u(y;x,t)\,\mathcal{G}(x,u(y;x,t))\, dy,
& \mbox{ in } \overline{\Omega} \times (0,T)
 ,\\[1.5 ex]
u(x,t)\leq h(x,t), & \mbox{ in }  (\Omega_J\setminus\overline{\Omega}) \times (0,T),
\\[1.5 ex]
u(x,0)\leq u_0(x),
& \mbox{ in }  \overline{\Omega},
\end{array}\right.
\end{equation}
with $u_0 (x) \in \mathcal{C} (\overline{\Omega}) $ and $h(x,t) \in    \mathcal{C}((\Omega_J\setminus \overline{\Omega})\times (0,T))$. \\
As usual, a regular supersolution is defined analogously by replacing ``$\leq$" with ``$\geq$".
Clearly, a regular solution is both a regular subsolution and a regular supersolution.
\end{definition}

\medskip

Next, we state our comparison principle.


\begin{theorem}{[Comparison Principle]}\label{C-P}
Let $u$ an $v$ be a regular subsolution and a regular supersolution of problem \eqref{P}, respectively, with boundary data $h_1 (x,t) $ and $h_2 (x,t)$ and initial data $u_0 (x)$ and $v_0 (x)$, respectively.
If $h_1 (x,t) \leq h_2 (x,t)$ in $\Omega_J\setminus \overline\Omega$ and $u_0 (x)\leq v_0 (x) $ in $\overline\Omega$, then
 $u\leq v$ in $\overline\Omega\times [0,T]$.
\end{theorem}

\begin{remark}
 The existence, uniqueness and comparison principle are also true relaxing the hypotheses on the  kernel $J(x-y)$ by considering a more  general one of the form  $K:\RR^N\times\RR^N\to \RR^+$ with compact support in $\Omega\times B(0,\rho)$, with $\rho >0$ such that
\[
0<\sup_{y \in B(0,\rho )}K(x,y)=R(x)\in L^\infty(\Omega).
\]
\end{remark}

The next result we want to prove relates solutions of local and nonlocal equations.
In order to do it, let us fix  a H\"older continuous function $\mu: \overline\Omega \to \RR$ with exponent $\alpha\in (0,1)$, and consider
\begin{equation}\label{G_mu}
\mathcal{G}_\mu(x,s)=1+\frac{\mu(x)\, s}{2(1+\mu(x)^2s^2)},\qquad(x,s)\in \overline\Omega\times \RR.
\end{equation}
The local problem we are interested in is the following
\begin{equation}\label{quasilinear}
\left\{
\begin{array}{ll}
v_t(x,t)=\Delta  v(x,t)+\mu(x) |\nabla   v(x,t)|^2
& \mbox{ in } \Omega \times (0,T),
\\[1.5 ex]
v(x,t)=h_0(x,t) & \mbox{ on } \partial \Omega \times (0,T),
\\[1.5 ex]
v(x,0)=v_0(x)
& \mbox{ in } \Omega\,.
\end{array}\right.
\end{equation}
Observe that if, for the same $0<\alpha<1$, we have   $\partial\Omega \in \mathcal C^{2+\alpha}$, $v_0\in \mathcal C^{1+\alpha}(\overline\Omega)$, $h\in \mathcal C^{1+\alpha, 1+\alpha/2}(\partial\Omega\times [0,T])$ with $v_0 $ and $h$ compatible (namely,  they are globally a $C^{1+\alpha, 1+\alpha/2}$ function of the parabolic boundary of the cylinder) and the equation holds up to the boundary, then Theorem 6.1 of Chapter V in \cite{LSU} provides a solution $v\in \mathcal C^{2+\alpha, 1+\alpha/2}(\overline\Omega\times(0,T])$.

Such a result becomes trivial if we assume $\mu(x) = \mu \in \RR$, after the Hopf--Cole transformation, since solutions of the heat equation satisfy the required regularity.

\medskip

We set here the definition of {\it classical solution}.

\begin{definition}
We say that  $v\in \mathcal C(\overline\Omega\times[0,T])\cap \mathcal C^{2+\alpha,1+\alpha/2}(\Omega\times(0,T))$ is
a \emph{classical solution} for the
Dirichlet problem \eqref{quasilinear} if it satisfies both the equations and the boundary and initial conditions in a  pointwise sense.

\end{definition}

\medskip

Consider now, for any $\varepsilon>0$  the rescaling nonlocal problem
\begin{equation}\label{P_e}
\left\{
\begin{array}{ll}
u_t^\varepsilon(x,t)=
\displaystyle \frac{C(x)}{\varepsilon^2} \  \displaystyle \int_{\Omega_{J_\varepsilon}} \negthickspace \negthickspace J_\varepsilon(x-y)u^\varepsilon(y;x,t)\,\mathcal{G}_\mu(x,u^\varepsilon(y;x,t))dy \
& \mbox{ in } \overline{\Omega} \times (0,T),
\\[1.5 ex]
u^\varepsilon(x,t)=h(x,t) & \mbox{ in }  (\Omega_{J_\varepsilon} \setminus\overline{\Omega}) \times (0,T),
\\[1.5 ex]
u^\varepsilon (x,0)=u_0(x)
&  \mbox{ in }  \overline{\Omega},
\end{array}
\right.
\end{equation}
where  $\mathcal{G}_\mu$ defined in \eqref{G_mu} and $C(x)$, $u_0$ and $h$ are  suitable measurable functions.

\medskip

Here we state our converging result.

\medskip

\begin{theorem}\label{principal}
Let $\Omega$ be a $\mathcal C^{2+\alpha}$, with $\alpha \in (0,1)$, bounded domain of $\RR^N$, $N\geq 1$, and let $v$ be a classical solution of the quasilinear problem \eqref{quasilinear} with
$h  \in \mathcal C^{1+\alpha} \left( \Omega_{J_\varepsilon} \negthickspace \setminus \negthickspace  \Omega \times (0,T]\right)$ such that $h\big|_{\partial \Omega \times (0,T)} = h_0 (x,t)$
and $v_0 \in \mathcal C^{1+\alpha} (\overline\Omega)$.
Assume that $J$ satisfies \rife{(H)} and that for a.e. $x$ in $\Omega$, $\mathcal{G} (x,s)$ is a $\mathcal{C}^{1+\alpha}$ function with respect to the $s$ variable such that  that   \rife{Gi} holds true.
For any  $\varepsilon > 0$, let $u^\varepsilon$ denote the solution to
\begin{equation}\label{P_ee}
\left\{
\begin{array}{ll}
u_t^\varepsilon(x,t)=
\displaystyle \frac{C(x)}{\varepsilon^2} \  \displaystyle \int_{\Omega_{J_\varepsilon}} \negthickspace \negthickspace J_\varepsilon(x-y)u^\varepsilon(y;x,t)\,\mathcal{G}(x,u^\varepsilon(y;x,t))dy  
& \mbox{ in } \overline{\Omega} \times (0,T),
\\[1.5 ex]
u^\varepsilon(x,t)=h(x,t) & \mbox{ in }  (\Omega_{J_\varepsilon} \setminus\overline{\Omega}) \times (0,T),
\\[1.5 ex]
u^\varepsilon (x,0)=v_0(x)
& \mbox{ in }  \overline{\Omega},
\end{array}
\right.
\end{equation}
with   $ C(x) ^{-1}=\frac{1}{2}C(J)\mathcal{G}(x,0)$ and  $\mu(x)=\displaystyle \frac{2\mathcal{G}_s^\prime(x,0)}{\mathcal{G}(x,0)}$ for any  a.e. $x\in \Omega$. Then we have
$$
\displaystyle \lim_{\varepsilon \to 0}\quad \sup_{t\in [0,T]} \Big\|u^\varepsilon( x,t)-v(x,t) \Big\|_{L^{\infty} (\Omega)} = 0\, .
$$
\end{theorem}

Let us stress that the same kind of result (as well as the existence, uniqueness and Comparison Principle one) can be proved in a more general framework. First of all, we might consider the same equation adding on the right hand side a (smooth enough) function. On the other hand, a more general kernel, that depends also on $y$ could be considered (see Remark \ref{4.3} for some more details). We decided to skip these generalizations in order to keep the paper more readable.

\medskip

The last type of results of this section deals with the asymptotic behavior of the solutions to \rife{P}. More precisely we prove, as it is usual for parabolic equations, that a solution of problem \rife{P} converges, for large times, to a stationary solution of the same problem.

\medskip

In order to avoid technicalities, we assume that the lateral condition is homogeneous, i.e. $h(x,t) \equiv 0$.

\medskip

 Here we state our result that asserts such a convergence, even if, under some additional hypotheses, we provide results on the rate of convergence (see Remark \ref{4.5} for more details).


\begin{theorem}\label{comp2}
For every $0 \leq u_0\in\mathcal C_0(\overline\Omega)$, the regular solution to problem
\begin{equation}\label{prob-hom}
\left\{
\begin{array}{ll}
  u_t(x,t) = \displaystyle \int_{\Omega_J}J(x-y)u(y;x,t) \ \mathcal G(x,u(y;x,t))\,dy
  & \mbox{ in } \overline{\Omega} \times (0,+ \infty),
\\[1.5 ex]
u(x,t)=0, & \mbox{ in }  \Omega_J\setminus \overline\Omega\times (0,+ \infty), \,\,t>0,\\[1.5 ex]
u(x,0)=u_0(x) & \mbox{ in } \overline\Omega,
\end{array}\right.
\end{equation}
satisfies $$\lim_{t\to\infty}u(x,t)=0 \qquad \mbox{  uniformly in } \quad \overline\Omega\,.$$
\end{theorem}

\begin{remark}
We want to stress that the hypothesis $u_0\geq0$ is not, in fact,  necessary, but we assume it just to let the proof easier.
\end{remark}

\medskip

Let us just point out that we have two special cases whose asymptotic behavior is well known in the local setting. If we assume that
\begin{align}\label{mu<0}
\exists \beta>0 \textrm{ : } \mathcal{G}(x,s)s\leq \beta s,\,\qquad  \forall s \in \RR, \mbox{ for a.e. } x \in \RR^N \,,
\end{align}
it corresponds to the {\it absorption case}, i.e. the case in which we have (at least) the same decay estimates as if $\mathcal{G}\equiv 1$.
In fact we can deduce (see Remark  \ref{4.5}) that in the absorption case the rate of convergence at $0$ is of exponential type.
On the other hand, if
\begin{align}\label{mu>0}
\exists \beta>0 \textrm{ : } \mathcal{G}(x,s)s\geq \beta s,\,\qquad  \forall s \in \RR, \mbox{ for a.e. } x \in \RR^N.
\end{align}
the result is more surprising since it correspond to the {\it reaction case}. In this framework it is crucial to deal with smooth solutions, since we exploit, in the proof,  the comparison principle.

\subsection{Cauchy problem}

This section deals with the Cauchy Problem related to \eqref{general}, that is
\begin{equation*}\label{C}
\left\{
\begin{array}{ll}
\negthickspace u_t(x,t)\negthickspace=\negthickspace\displaystyle\int_{\RR^N}\negthickspace\negthickspace\negthickspace J(x-y)u(y;x,t) \ \mathcal{G}(x,u(y;x,t))\,dy
& \mbox{ in } \RR^N \times (0,T),\\[2.0 ex]
u(x,0)=u_0(x)
& \mbox{ in } \RR^N  \,,
\end{array}\right.
\tag{$C$}
\end{equation*}
with $\mathcal{G}  $ as in \eqref{Gi}, $J$ as in \rife{(H)} and $u_0 \in \mathcal{C} (\RR^N)$. First let us give the notion of solution.
\begin{definition}
Given $u_0 \in \mathcal{C}(\RR^N)  $ we define a solution of problem \eqref{C} as a function $u \in \mathcal{C}\left([0,T);\mathcal{C}(\RR^N) \right)$ such that it satisfies
$$
u(x,t)=\int_0^t\ir\negthickspace\negthickspace\negthickspace J(x-y)u(y;x,\tau)\mathcal{G}(x,u(y;x,\tau))dyd\tau+u_0(x) \qquad \mbox{ in } \RR^N \times (0,T).
$$
Consequently, due to the integral expression above,  $u \in \mathcal{C}^1\left((0,T);\mathcal{C}(\RR^N) \right)$.
\end{definition}

\medskip

The first result we want to present in this framework deals with the existence of a bounded solution.

\medskip

\begin{theorem}{[Existence]}\label{existence_Cauchy} For every continuous and bounded initial data $u_0$ there exists a unique solution $u \in \mathcal{C}\left([0,T);\mathcal{C}(\RR^N)\cap L^\infty(\RR^N)\right)$ of problem \eqref{C}.
\end{theorem}

We continue this section proving the comparison principle for our problem. For this purpose, we first set the notion of sub and supersolution.
\begin{definition} A function $u \in \mathcal{C}^0\left([0, T ),\mathcal{C}(\RR^N)\right) \cap \mathcal{C}^1\left((0, T ),\mathcal{C}(\RR^N)\right)$ is a subsolution of problem \eqref{C} if it satisfies
\begin{equation*}
\left\{
\begin{array}{ll}
\negthickspace u_t(x,t)\negthickspace \leq \negthickspace\displaystyle\int_{\RR^N}\negthickspace\negthickspace\negthickspace J(x-y)u(y;x,t)\mathcal{G}(x,u(y;x,t))dy,
& \mbox{ in } \RR^N \times (0,T),\\[10pt]
u(x,0)\leq u_0(x),
& \mbox{ in } \RR^N.
\end{array}\right.
\end{equation*}
As usual, a supersolution is defined analogously by replacing $``\leq"$ by $``\geq"$.
\end{definition}

Next we state the comparison principle in this framework.

\begin{theorem}{[Comparison Principle]}\label{C-P_Cauchy} Let $u,v$ be a subsolution and supersolution respectively of problem \eqref{C} with initial data $u_0\in \mathcal{ C } (\RR^N) \cap L^{\infty} (\RR^N)$ and $v_0\in \mathcal{ C } (\RR^N) \cap L^{\infty} (\RR^N)$, respectively, such that  $u_0 \leq v_0$   in $\RR^N$. Then $u\leq v$ in $\RR^N \times (0,T)$.
\end{theorem}

Now, we prove that given a classical solution  (i.e., $v \in \mathcal{C}^{2+\alpha,1+\alpha/2}\left(\RR^N \times [0,T]    \right)$) of the parabolic problem with a quadratic gradient term of the form
\begin{equation}\label{quasilinear_Cauchy}
\left\{
\begin{array}{ll}
v_t(x,t)=\Delta  v(x,t)+\mu(x) |\nabla   v(x,t)|^2
& \mbox{ in } \RR^N \times (0,T) \\[1.5 ex]v(x,0)=v_0(x)
& \mbox{ in } \RR^N ,
\end{array}\right.
\end{equation}
with  $v_0 \in \mathcal{C}(\RR^N)\cap L^\infty(\RR^N)$ and $\mu (x) \in \mathcal{C}^{\alpha}\left(\RR^N     \right) \cap L^{\infty} (\RR^N)$, it can be approximated by a solution of the nonlocal problem
\begin{equation}\label{P_e_Cauchy}
\left\{
\begin{array}{ll}
\negthickspace u_t^\varepsilon=\displaystyle \frac{C(x)}{\varepsilon^2}  \displaystyle \ir \negthickspace \negthickspace J_\varepsilon(x-y)u^\varepsilon(y;x,t)\,\mathcal{G}(x,u^\varepsilon(y;x,t))dy
& \mbox{ in } \RR^N \times (0,T),\\[1.5 ex]
\negthickspace u^\varepsilon (x,0)=v_0(x),
& \mbox{ in } \RR^N ,
\end{array}
\right.
\end{equation}
such that $\displaystyle \frac{2\mathcal{G}_s^\prime(x,0)}{\mathcal{G}(x,0)}=\mu(x)$. As usual $C(x)^{-1}=\frac{1}{2}C(J)\mathcal{G}(x,0)\neq 0$ and  $J_\varepsilon (s)=\frac{1}{\varepsilon^N}J(\frac{s}{\varepsilon})$.


\medskip

\begin{theorem}\label{principal_Cauchy}
Let $v$ be a classical solution of quasilinear differential equation \eqref{quasilinear_Cauchy}. Let, for a given $\varepsilon > 0$, $u^\varepsilon$ be the solution to \eqref{P_e_Cauchy}, with the same initial datum $v_0 \in \mathcal{C}(\RR^N)\cap L^\infty(\RR^N)$. Then, we have

$$
\lim_{\varepsilon \to 0}  \quad \sup_{t\in [0,T]} \big\| u^\varepsilon(\cdot,t)-v(\cdot,t)\big\|_{L^\infty(\RR^N)}= 0 \,.
$$
\end{theorem}

\medskip

Finally, we study the asymptotic behavior of the solutions associated to the Cauchy problem.

Our result depends on the nature of $\mathcal{G}$, i.e. if it is of absorption or reaction type.

%
%
%
%
%
%

\medskip


Summarizing, we obtain the following results:

\bigskip

%
%

\begin{theorem}\label{asymp Cauchy}
 For $N\geq 1$, Llet $u$ be a solution of Cauchy problem \eqref{C}   satisfying \eqref{mu<0}
and positive initial datum  $u_0\in L^1(\RR^N)\cap L^\infty(\RR^N) \cap \mathcal{C} (\RR^N) $. Then there exists $C=C(J,N,\beta,q)>0$ such that
\[
\|u(\cdot,t)\|_{L^q (\RR^N)} \leq C\|u_0\|_{L^1 (\RR^N)}\,t^{-\frac{N}{2}\left(1-\frac{1}{q}\right)}, \textrm{ for any } \,q\in[1,\infty),
\]
for  $t$ sufficiently large.
\end{theorem}

\begin{theorem}\label{L^2 mu>0}
For $N\geq 1$, let $u$ be a solution of of Cauchy problem \eqref{C} with $\mathcal{G}\equiv \mathcal{G}_\mu$, $0\leq \mu \in L^\infty(\RR^N)$ and positive initial datum   $u_0\in L^1(\RR^N)\cap L^\infty(\RR^N)\cap \mathcal{C}(\RR^N)$ satisfying
\begin{equation}\label{theta}
\|u_0\|_{L^\infty(\RR^N)}\|\mu\|_{L^\infty(\RR^N)}<1  \,.
\end{equation}
Then,
\[
\|u(\cdot,t)\|_{L^2(\RR^N)}^2\leq \tilde{C}\|u_0\|_{L^1(\RR^N)}\, t^{-\frac{N}{2}},
\]
for some $\tilde{C}=\tilde{C}(\|\mu\|_{L^\infty(\RR^N)},\|u_0\|_{L^\infty(\RR^N)},N,J)>0$ and for  $t$ sufficiently large.
\end{theorem}

\section{Preliminaries}

{\bf Notation. }   Throughout this paper, we always use the following notation:\\
 we denote in a short way  $u(y;x,t)=u(y,t)-u(x,t)$. Moreover the time variable will always get values between $0$ and $T$, with $T>0$. As far as the kernel $J$ is concerned, we assume that it is defined as in $\eqref{(H)}$ and such that  $\mathcal{G}$   satisfies  \eqref{Gi} and    $C=2C(J)^{-1}$,
      $J_\varepsilon (s)=\frac{1}{\varepsilon^N}J(\frac{s}{\varepsilon})$.

\medskip

As far as the the function $\mathcal{G}(x,s)$ is concerned, we
observe  that, for a function $\mathcal{G}$ differentiable with respect to $s$ we have, thanks to  the Mean Value Theorem, that
\begin{equation}\label{Gi2}
\alpha_1\leq\mathcal{G}_s^\prime(x,s)s+\mathcal{G}(x,s)\leq \alpha_2, \qquad \hbox{for any  }s\in\RR \hbox{ a.e. in  }x\in \RR^N.
\end{equation}

Moreover, if  $\mathcal{G}$ is differentiable with respect to $s$, condition \eqref{Gi} is equivalent to define $\psi:\RR^N\times \RR \times \RR \to \RR$ with
$$
0<\alpha_1\leq \psi(x,s,\sigma)\leq \alpha_2 \qquad \mbox{ for a.e. x } \in \Omega, \quad \forall s,\sigma \in \RR,
$$ such that
\begin{equation}\label{h}
\psi (x,s,\sigma)=
\left\{
\begin{array}{ll}
\displaystyle
 \frac{\mathcal{G}(x,s)s-\mathcal{G}(x,\sigma)\sigma  }{s-\sigma} \qquad & \mbox{ if } \quad s\neq \sigma\,,
\\[1.5 ex] \displaystyle
  \mathcal{G}_s^\prime(x,s)s+\mathcal{G}(x,s)
& \mbox{ if } \quad s= \sigma.
\end{array}
\right.
\end{equation}

We also remark that, in particular, condition \eqref{Gi} implies $\mathcal{G}(x,0)\neq 0$ for any  $x\in \RR^N$.

\medskip

Here, we state the following technical result  which   allow us to see that the function defined in \eqref{G_mu} satisfies the basic condition \eqref{Gi}.

\begin{proposition}\label{m-condition}
Let $p,q$ and $k$ be real numbers, then the following properties hold true
$$
\frac{3}{4}\leq 1+\frac{kp}{2(1+k^2p^2)}\leq \frac{5}{4}\,,
$$
$$
p  \left[1+\frac{k p}{2(1+k^2p^2)}\right] -q  \left[1+\frac{k q}{2(1+k^2q^2)}\right]= (p-q) \left[1+\frac{k(p+q)}{2(1+k^2p^2)(1+k^2q^2) } \right]\,,
$$
$$
1-\frac{3\sqrt{3}}{16}\leq 1+\displaystyle \frac{k(p+q)}{2(1+k^2p^2)(1+k^2q^2)} \leq 1+\frac{3\sqrt{3}}{16}\,.
$$
Moreover, for any measurable function $\mu:\RR^N\to \RR$, the function defined by $\displaystyle \mathcal{G}_\mu(x,s)=1+\frac{\mu(x) s}{2(1+\mu(x)^2s^2)}$ satisfies the following conditions
\begin{enumerate}
\item[(i)] $\left(1-\frac{3\sqrt{3}}{16}\right)(s-\sigma)\le \mathcal{G}_\mu(x,s)s-\mathcal{G}_\mu(x,\sigma)\sigma\le \left(1+\frac{3\sqrt{3}}{16}\right)(s-\sigma),$
 \,\, for $s> \sigma$, $\,x\in\RR^N$;
  \item[(ii)] if  $\mu \geq 0$, then $\mathcal{G}_\mu(x,s)s\leq s$, for any $(x,s)\in \RR^N\times\RR$;
\item[(iii)] if   $\mu \leq 0$    then $\mathcal{G}_\mu(x,s)s\geq s$, for any $(x,s)\in \RR^N\times\RR$.
\end{enumerate}
\end{proposition}

\begin{proof} The first two inequalities are  straightforward while for the third one   we just remark that the function given by
\[
f(x,y)=\frac{|x|+|y|}{(1+x^2)(1+y^2)}
\]
attains its maximum $\frac{3\sqrt 3}8$ at the point $(\frac13,\frac13)$.\\
Now,   $(i)$ is a consequence of the previous inequalities. Conditions $(ii)$ and $(iii)$ follow by the fact that
\begin{equation*}
\left\{
\begin{array}{lr}
\frac{3}{4}\leq G_\mu(x,s)\leq 1, & \mbox{ if } (x,s)\in \RR^N\times [0,\infty),
\\[1.5 ex]
1\leq G_\mu(x,s)\leq \frac{5}{4}, & \mbox{ if }  (x,s)\in \RR^N\times (-\infty,0],
\end{array}
\right.
\end{equation*}
for $\mu(x)\leq0$, and
\begin{equation*}
\left\{
\begin{array}{lr}
1\leq G_\mu(x,s)\leq \frac{5}{4}, & \mbox{ if }  (x,s)\in \RR^N\times [0,\infty),
\\[1.5 ex]
\frac{3}{4}\leq G_\mu(x,s)\leq 1, & \mbox{ if }  (x,s)\in \RR^N\times (-\infty,0],
\end{array}
\right.
\end{equation*}
for $\mu(x) \geq 0$.
\end{proof}

\begin{remark}
Let us stress that in the above result we only assume that $\mu(x)$ is measurable, without any hypotheses on its regularity.
\end{remark}

 \begin{lemma} \label{qu}    Let $q\geq1$, there exists $c(q)>0$ such that
\begin{equation}\label{B}
(a-b)(a^{q-1}-b^{q-1})\geq c(q)\,(a^{\frac{q}2}-b^{\frac{q}2})^2, \hspace{0,4cm} \textrm{ for any  }a,b\geq 0.
\end{equation}
\end{lemma}
\begin{proof} Without loss of generality we assume $a > b$. Therefore, \eqref{B} is equivalent to prove that
the  function
\[
F(\theta)= \frac{(1-\theta)(1-\theta ^{q-1})}{(1-\theta^{\frac{q}2})^2}\qquad  \theta \in [0,1),
\]
is bounded below by a  $0<c(q)$, being $\theta=b/a$.

The result just follows by computing the derivative of $F$ and noticing that it is decreasing. Hence the minimum of $F$ is achieved at $\theta=1$, and
$
\displaystyle \lim_{\theta\to1^-} F (\theta)= 4 \frac{q-1}{q^2}\,.
$ \end{proof}

\section{Proofs of results about Dirichlet Problem} \label{4}

We start by proving the existence result.

\begin{proof}[Proof of Theorem \ref{existence}]
i) Fixed an arbitrary $T>0$, we set the Banach space $X_T=\mathcal{C}([0,T];L^1(\Omega))$ endowed with norm
\begin{equation}\label{norma}
|||v|||=\max_{0\leq t \leq T}e^{-Mt}\|v(\cdot,t)\|_{L^1(\Omega)} \,,
\end{equation}
for some $M \geq \tilde{C}=\alpha_2\,\|J\|_{L^{\infty}(\Omega)} \left(|\Omega|+|\textrm{supp}(J)|\right)$.

\medskip

  Let $\mathcal{T}:X_T\to X_T$ be the operator defined by
\[
\mathcal{T}(v)(x,t)=\int_0^t \int_{\Omega_J} J(x-y)v(y;x,\tau)\mathcal{G}(x,v(y;x,\tau))\,dy\,d\tau+u_0(x),
\]
with $v(y,t)=h(y,t)$ for $y\in \Omega_J\setminus \Omega$. Then, we prove the existence and uniqueness of solutions of \eqref{P} via the standard Banach contraction principle applied to the operator $\mathcal{T}$. In this way, using Fubini's Theorem and since $\mathcal{G}$ is bounded, we obtain
\begin{equation}\begin{array}{c} \label{explained}
\dys  \|\mathcal{T}(v(\cdot,t))\|_{L^1(\Omega)} \leq \|u_0\|_{L^1(\Omega)}+\alpha_2\int_0^t\int_\Omega \int_{\Omega_J} J(x-y)|v(y;x,\tau)|\,dy\,dx\,d\tau \\\dys
 \leq \|u_0\|_{L^1(\Omega)}+\alpha_2 \int_0^t\left(\int_\Omega \int_{\Omega_J} J(x-y)|v(y,\tau)|\,dy\,dx\right.
+\left. \int_\Omega \int_{\Omega_J} J(x-y)|v(x,\tau)|\,dy\,dx  \right)d\tau
\\
\dys \leq \|u_0\|_{L^1(\Omega)}+\frac{\tilde{C}_1}{M}\left(e^{Mt}-1\right)|||v|||+\tilde{C}_2\,,
\end{array}\end{equation}
where $\tilde{C}_1=\alpha_2\|J\|_{L^\infty(\RR^N)} \left(|\Omega|+|\textrm{supp}(J)|\right)$ and $\tilde{C}_2=\alpha_2\|J\|_{L^\infty(\RR^N)} |\Omega|\, \|h\|_{L^1((\Omega_J\setminus \Omega)\times (0,\infty))}$\,. Therefore
\begin{align*}
|||\mathcal{T}(v)||| & \leq \max_{0\leq t \leq T}\left(e^{-Mt}(\|u_0\|_{L^1(\Omega)}+\tilde{C}_2) +\frac{\tilde{C}_1}{M}\left(1-e^{-Mt}\right)|||v|||  \right)
\leq \|u_0\|_{L^1(\Omega)}+\tilde{C}_2+\frac{\tilde{C}_1}{M}|||v|||\,.
\end{align*}
 Hence, $\mathcal{T}$ maps $X_{T}$ into itself. Note that all the involved constants do not depend on the value $T$.

  Now,
 by virtue of \eqref{Gi}, we can assert that for every $w,z \in X_T$
\begin{align*}
\big|(\mathcal{T}(w)-\mathcal{T}(z))(x,t)\big|&\le\int_0^t\int_{\Omega_J} J(x-y)\big|w(y;x,\tau)\mathcal{G}(x,w(y;x,\tau))- z(y;x,\tau)\mathcal{G}(x,z(y;x,\tau))  \big|\,dy\,d\tau
\\
&\leq \alpha_2 \int_0^t\int_{\Omega_J} J(x-y)\left|w(y;x,\tau)-z(y;x,\tau)\right|\,dy\,d\tau.
\end{align*}
Therefore, arguing as in \eqref{explained}, we get
\[
\|\mathcal{T}(w)-\mathcal{T}(z)\|_{L^1(\Omega)}\leq\frac{\tilde{C}_1}{M}\left(e^{Mt}-1\right)|||w-z|||\,.
\]
Thus, since $M>\tilde{C}$, we get
\[
|||\mathcal{T}(w)-\mathcal{T}(z)|||\le\vartheta |||w-z|||\,,
\]
with $0<\vartheta <1$.
Hence $\mathcal{T}$ is a contraction and by the Banach's Fixed Point Theorem there exists a unique $u\in X_{T}$ such that $\mathcal{T}(u)=u$, i.e., consequently we get local existence and uniqueness of problem $\eqref{P}$ for $0\leq t \leq T$. Moreover, since this argument is independent of the value $T$, we obtain a unique solution $u \in \mathcal{C}([0,\infty);L^1(\Omega))$ of problem $\eqref{P}$.

\medskip

ii) For the second part it is sufficient to change the definition of $||| \cdot |||$ in \eqref{norma}, replacing $L^1 (\Omega)$ with $\mathcal  C(\overline\Omega)$.
The regularity of $u_t$ easily follows by using the equation solved by $u$.
\end{proof}

\medskip

Next we deal with the proof of the comparison principle.

\medskip

\begin{proof}[Proof of Theorem \ref{C-P}]
We denote by $w=u-v$. Obviously $w\in \mathcal C(\overline\Omega\times [0,T])$, $w_t\in \mathcal  C(\overline\Omega\times (0,T))$ and it satisfies
$$
\left\{
\begin{array}{ll}
  w_t(x,t) \leq    \displaystyle \int_{\Omega_J}  J(x-y)(w(y,t)-w(x,t))  \psi (w(y;x,t)) \,dy, & \mbox{ in } \overline\Omega\times (0,T) ,\\[1.5 ex]
w(x,t)\leq0, & \mbox{ in } \Omega_J\setminus \overline\Omega \times (0,T),\\[1.5 ex]
w(x,0)\leq0, & \mbox{ in }\overline\Omega,
\end{array}\right.
$$
where $\psi$ is the function defined in \rife{h}.

Assume by contradiction that $w(x,t)$ is positive at some point $(\tilde{x},\tilde{t})$ that, without loss of generality, we can assume that belongs to $\Omega\times (0,T]$.
Thus, by the continuity of $u$ and $v$, there exists a  $\delta>0$ such that $w(\tilde{x},\tilde{t}) -\delta \tilde{t}>0$.
Let us denote by $(x_0,t_0)$ the maximum   point of $w(x,t) - \delta t $ which is, by construction, positive. Consequently
being $u_t$ continuous in $\Omega\times (0,T)$, we have that
$$
w_t (x_0,t_0) - \delta \geq 0  \,.
$$
On the other hand, plugging it into the equation in \eqref{p}, we get
$$
\begin{array}{c}
\displaystyle
w_t(x_0,t_0) \leq \int_{\Omega_J }J(x_0-y)\big(w(y,t_0)-w(x_0,t_0)\big) \psi \big(w(y,t_0)-w(x_0,t_0)\big) dy
\\\displaystyle
=
\int_{  \Omega }J(x_0-y)\big((w(y,t_0)- \delta t_0)  - (w(x_0,t_0)- \delta t_0) \big) \psi \big( (w(y,t_0)- \delta t_0) -(w(x_0,t_0)- \delta t_0)\big) dy
\\[1.5 ex]\displaystyle
+\int_{\Omega_J \setminus \Omega }J(x_0-y)\big((w(y,t_0)- \delta t_0)  - (w(x_0,t_0)- \delta t_0) \big) \psi \big((w(y,t_0)- \delta t_0)  - (w(x_0,t_0)- \delta t_0)\big) dy\,,
\end{array}
$$
and the last two integrals are nonpositive.
Indeed as far as the first one is concerned, we observe that $(x_0, t_0)$ is a maximum point, while $\psi $ is positive; moreover outside $\Omega$  we use that the boundary condition is negative and that  $w(x_0,t_0)- \delta t_0>0$, as well as the positivity of $\psi$. Hence we get a contradiction.
\end{proof}

Our goal is now to get a proof of Theorem \ref{principal}. Here, we start with a preliminary Lemma.

\begin{lemma}\label{approx}
Let $u \in \mathcal{C}^{2+\alpha,1+\alpha /2}\left(\RR^N\times [0,T]\right)$, $\mathcal{G}(x,s)$ a $\mathcal{C}^{1+\alpha}$ function with  respect to variable $s$  such that  $\mathcal{G}(x,0)\neq 0$ for a.e.  $x\in \Omega$, and let  $\mathcal{L}_\varepsilon$ be the following operator
\begin{equation}\label{L_e}
\mathcal{L}_\varepsilon(u(x,t))=\displaystyle \frac{ C(x)}{\varepsilon^2} \displaystyle \ir J_\varepsilon(x-y)\,u(y;x,t)\,\mathcal{G}(x,u(y;x,t))\,dy  ,
\end{equation}
where $ \frac{1}{C(x)}  =\frac{1}{2}C(J)\ \mathcal{G}(x,0)$. Then, $\exists \,c=c(T)>0$ such that, $\forall \varepsilon>0$
$$
\sup_{t\in [0,T]}  \Big\|\mathcal{L}_\varepsilon(u(x,t))- \Delta  u(x,t)-\mu(x) |\nabla   u(x,t)|^2  \Big\|_{{L^{\infty} (\Omega)}}\leq c \ \varepsilon^\alpha
$$
where $\mu(x) =\displaystyle \frac{2\mathcal{G}_s^\prime(x,0)}{\mathcal{G}(x,0)}$.
\end{lemma}

\begin{remark}
Observe  that the integral expression above vanishes outside of $\Omega_{J_\varepsilon}=\Omega +\varepsilon \supp(J)$. In this way, $h$ is only needed to be prescribed in $\Omega_{J_\varepsilon} \setminus \overline\Omega$. Observe also that, thanks to the hypothesis of Theorem \ref{principal} we use, in the following, that
$$
h(x,t)= h_0 (x,t)+ O (\varepsilon) \qquad \mbox{ in } \ \Omega\setminus \Omega_{J_\varepsilon}\,.
$$
\end{remark}

\begin{proof}
In order to compute $\mathcal{L}_\varepsilon (u(x,t))$ we make the  change of  variables $y=x-\varepsilon z$, and we get
\begin{equation}\label{123}
\mathcal{L}_\varepsilon(u(x,t))=\displaystyle \frac{ C(x)}{\varepsilon^2} \,  \displaystyle \ir  J(z)u(x-\varepsilon z;x,t)\,\mathcal{G}(x,u(x-\varepsilon z;x,t))dz  .
\end{equation}
Moreover by Taylor formula we have that $\mathcal{G}(x,\delta)=\mathcal{G}(x,0)+\mathcal{G}_s^\prime(x,0)\delta+O(\delta^{1+\alpha})$, and
\[
u(x-\varepsilon z;x,t)=-\varepsilon\sum_i\frac{\partial u(x,t)}{\partial x_i}z_i+\frac{\varepsilon^2}{2}\sum_{i,j}\frac{\partial^2u(x,t)}{\partial x_i \partial x_j}z_iz_j +O\left(\varepsilon^{2+\alpha}\right),
\]
Consequently

\begin{equation}\label{tay}
\mathcal{L}_\varepsilon(u(x,t))=S_1(x,t)+S_2(x,t)+S_3(x,t)
\end{equation}
 being
\begin{align*}
S_1(x,t) &=\displaystyle \frac{C(x) \mathcal{G}(x,0)}{\varepsilon^2} \displaystyle \ir  J(z)u(x-\varepsilon z;x,t)\,dz,
\\[1.5 ex]S_2(x,t) &=\displaystyle \frac{ C(x) \mathcal{G}_s^\prime(x,0)}{\varepsilon^2} \displaystyle \ir  J(z)u(x-\varepsilon z;x,t)^2\,dz,
\\[1.5 ex]S_3(x,t) & = \displaystyle \frac{ C(x)}{\varepsilon^2} \displaystyle \ir  J(z)u(x-\varepsilon z;x,t)^{2+\alpha}\,dz = O (\varepsilon^\alpha).
\end{align*}
First, we deal with  $S_1(x,t)$ and we obtain
\begin{align}\label{S1}
\nonumber S_1(x,t) &=-\frac{C}{\varepsilon}\sum_i\frac{\partial u(x,t)}{\partial x_i}\ir J(z)z_idz+C(J)^{-1}\sum_{i,j}\frac{\partial^2u(x,t)}{\partial x_i \partial x_j}\ir J(z)z_iz_j +O\left(\varepsilon^{\alpha}\right)
\\
&=\sum_{i,j}\frac{\partial^2u(x,t)}{\partial x_i \partial x_j}+O\left(\varepsilon^{\alpha}\right),
\end{align}
using in the last equality that $J$ is radially symmetric, that is, $\ir J(z)z_idz=0$ and
$$\ir \negthickspace \negthickspace J(z)z_iz_j\, dz=0 \qquad \mbox{  if  i $\neq$ j }.$$
In order to compute $S_2(x,t)$, using the expansion of
$u(x-\varepsilon z;x,t) $
up to the  first order, we get
\begin{equation}\label{S2}
\begin{array}{c}
\displaystyle S_2(x,t) =\displaystyle \frac{ C(x) \mathcal{G}_s^\prime(x,0)}{\varepsilon^2} \displaystyle \ir  J(z)\left( -\varepsilon\sum_i\frac{\partial u(x,t)}{\partial x_i}z_i+O\left(\varepsilon^{1+\alpha}\right)\right)^2dz
  \\
\displaystyle = C(x) \mathcal{G}_s^\prime(x,0) \sum_{i,j} \frac{\partial u(x,t)}{\partial x_i}\frac{\partial u(x,t)}{\partial x_j}\ir J(z)z_iz_jdz +O\left(\varepsilon^\alpha  \right)
  =\displaystyle \frac{2\mathcal{G}_s^\prime(x,0)}{\mathcal{G}(x,0)}\sum_i\left(\frac{\partial u(x,t)}{\partial x_i}  \right)^2+O\left(\varepsilon^\alpha  \right),
\end{array}
\end{equation}
using again, in the last equality, that $J$ is radially symmetric. Finally, setting $u(x-\varepsilon z;x,t)=O(\varepsilon)$, we obtain that $S_3(x,t)=O\left(\varepsilon^\alpha  \right)$ and gathering together \rife{tay} with \eqref{S1} and \eqref{S2}, we deduce that  \eqref{123}  becomes
$$
\mathcal{L}_\varepsilon(u(x,t))=\Delta u(x,t)+\mu(x) |\nabla   u(x,t)|^2 +O\left(\varepsilon^{\alpha}\right)
$$
concluding the proof.
\end{proof}

\begin{remark}\label{4.3}
Arguing as the in the proof of the above Lemma, we can state the following assertion: the operator defined as
\begin{equation*}
\tilde{\mathcal{L}}_\varepsilon(u(x,t))=\displaystyle \frac{ C(x)}{\varepsilon^2}  \displaystyle \ir J_\varepsilon(x-y)\,u(y;x,t)\,\mathcal{G}({ y},u(y;x,t))\,dy  ,
\end{equation*}
converges uniformly in $[0,T]\times \overline \Omega$, as $\varepsilon \to 0$, to the operator
\begin{equation*}
\Delta u(x,t)+\nabla_{ y}\eta (x,0)\nabla  u(x,t)+\eta_s^\prime(x,0)|\nabla  u(x,t)|^2,
\end{equation*}
being $\eta(x,s)=\log G(x,s)^2$. Therefore, the role of the variables is not symmetric.
\\
\end{remark}
\begin{remark}
Let us recall that given $\mu:\RR^N\to \RR$, then $\mathcal{G}_\mu(x,s)$ defined in \eqref{G_mu} satisfies $ \frac{2\mathcal{G}_s^\prime(x,0)}{\mathcal{G}(x,0)}=\mu(x)$, for any  $x\in \RR^N$.
\end{remark}

  Now, we prove the main result of this section. That is, classical solutions of \eqref{quasilinear} can be approximated by solutions of problem \eqref{P_e} which in a general setting reads as follows,

\begin{proof}[Proof of Theorem \ref{principal}]
Let $\tilde{v}$ be a $\mathcal{C}^{2+\alpha, 1+\alpha/2}\left(\RR^N\times [0,T]  \right)$ extension of $v$, the solution to \eqref{quasilinear}.
Denote by $h(x,t)=\tilde{v}(x,t)$ for any  $(x,t)\in (\RR^N\backslash \Omega)\times(0,T]$. Then $h$ is smooth and $h(x,t)=h_0(x,t)$ if $x \in \partial \Omega$ and we get
\begin{equation}\label{G}
h(x,t)=h_0 (x,t)+O(\varepsilon), \hspace{0,3cm} \textrm{for} \,\, x\in\Omega_{J_\varepsilon} \setminus \Omega.
\end{equation}

Observe that $\tilde{v}$ verifies
\begin{equation*}
\left\{
\begin{array}{ll}
\tilde{v}_t(x,t)=\Delta  \tilde{v}(x,t)+\mu(x) |\nabla   \tilde{v}(x,t)|^2
&  \mbox{ in }   \Omega,
\\[1.5 ex]
\tilde{v}(x,t)=h(x,t) & \mbox{ in }  (\Omega_{J_\varepsilon} \setminus \Omega) \times  (0,T),
\\[1.5 ex]\tilde{v}(x,0)=v_0(x)
& \mbox{ in }  \Omega.
\end{array}\right.
\end{equation*}

 \medskip

 Theorem \ref{existence} asserts that, for any given $\varepsilon>0$, there exists a unique $u^\varepsilon$ which is solution to \eqref{P_ee}.

  Set $w^\varepsilon:=\tilde{v}-u^\varepsilon$, which satisfies
\begin{equation}\label{wi}
\left\{
\begin{array}{ll}
w^\varepsilon_t(x,t)=\Delta  \tilde{v}(x,t)+\mu(x) |\nabla   \tilde{v}(x,t)|^2-\mathcal{L}_\varepsilon(u^\varepsilon(x,t))
&   \mbox{ in }   \Omega\times (0,T),
\\[1.5 ex]
w^\varepsilon(x,t)=0 & \mbox{ in }  (\Omega_{J_\varepsilon} \setminus \Omega) \times  (0,T),
\\[1.5 ex]
w^\varepsilon(x,0)=0
&  \mbox{ in }   \Omega.
\end{array}\right.
\end{equation}
By using condition \eqref{h}, we set
\begin{align*}
\mathcal{M}_\varepsilon(w^\varepsilon(x,t)) & :=\mathcal{L}_\varepsilon(\tilde{v}(x,t))-\mathcal{L}_\varepsilon(u^\varepsilon(x,t))
\\
& =\frac{C(x)}{\varepsilon^2}\int_{\Omega_{J_\varepsilon}} J_\varepsilon(x-y)\,\psi\left(x,\tilde v(y;x,t),u^\varepsilon(y;x,t)\right)\,w^\varepsilon(y;x,t)dy.
\\
\Lambda_\varepsilon(\tilde v(x,t))&:= \Delta  \tilde{v}(x,t) +  \mu(x) |\nabla   \tilde{v}(x,t)|^2    - \mathcal{L}_\varepsilon(\tilde{v}(x,t)).
\end{align*}
In this way, we replace equation \eqref{wi} by the following
\begin{equation}\label{sub-sup}
\left\{
\begin{array}{ll}
  w^\varepsilon_t(x,t)  = \Lambda_\varepsilon(\tilde v(x,t))+\mathcal{M}_\varepsilon(w^\varepsilon(x,t)),\,  & \mbox{ in }   \Omega\times (0,T),\\[1.5 ex]
  w^\varepsilon(x,t)= 0 \,, & \mbox{ in }  (\Omega_{J_\varepsilon} \setminus \Omega) \times  (0,T),
\\[1.5 ex]
w^\varepsilon(x,0)=0,
& \mbox{ in }   \Omega.
\end{array}\right.
\end{equation}
We begin by proving that for $K_1,K_2>0$ sufficiently large, $\overline{w}(x,t)=K_1\varepsilon^\alpha t+K_2\varepsilon$ is a supersolution of \eqref{sub-sup}. Indeed, taking into account Lemma \ref{approx} and  that $\mathcal{M}_\varepsilon(\overline{w}(x,t))=0$, we obtain
$$
\overline{w}_t(x,t)=K_1\varepsilon^\alpha \geq \Lambda_\varepsilon(\tilde v(x,t))=\Lambda_\varepsilon(\tilde v(x,t))+\mathcal{M}_\varepsilon(\overline{w}(x,t)),
$$
for $x\in \Omega$, $t\in (0,T]$. Moreover, $\overline{w}(x,0)>0$ and by \eqref{G},  we have that $\overline{w}(x,t)\geq K_2\varepsilon\geq O(\varepsilon)$, for $x\in\Omega_{J_\varepsilon} \setminus \Omega$ and $t\in (0,T]$. Consequently, $\overline{w}$  is a supersolution of \eqref{sub-sup}.
\\
 Now, by the comparison principle stated in Theorem \ref{C-P}, we get
\begin{equation}\label{first-inequality}
\tilde{v}-u^\varepsilon \leq K_1\varepsilon^\alpha t+K_2\varepsilon.
\end{equation}
By the other hand, similar arguments applied to the case $\underline{w}=-\overline{w}$ leads us to assert that $\underline{w}$
is a subsolution of \eqref{sub-sup} and using again the comparison principle we obtain
\begin{equation}\label{second-inequality}
\tilde{v}-u^\varepsilon \geq -K_1\varepsilon^\alpha t-K_2\varepsilon.
\end{equation}
Hence, by virtue of \eqref{first-inequality} and \eqref{second-inequality}
$$
\sup_{t\in [0,T]} \|u^\varepsilon(\cdot,t)-v(\cdot,t)\|_{L^{\infty} (\Omega)} \leq K_1\varepsilon^\alpha T+K_2\varepsilon ,
$$
that vanishes as $\varepsilon$ goes to  $0$.
\end{proof}


Here, we deal with the asymptotic behavior of the solution. In order to prove the main result (i.e. Theorem \ref{comp2}), we start with an intermediate  result.



\begin{theorem}\label{comp1}
Given $\lambda\ne0$, consider the problem
\begin{equation}\label{prob-comp}
\left\{
\begin{array}{ll}
  u_t(x,t) = \displaystyle \int_{\Omega_J}J(x-y)\mathcal G(x,u(y;x,t))u(y;x,t)\,dy, & x\in \overline\Omega, \,\,t>0,\\[1.5 ex]
u(x,t)=0, & x\in \Omega_J\setminus \overline\Omega, \,\,t>0.\\[1.5 ex]
u(x,0)=\lambda, & x\in \Omega.
\end{array}\right.
\end{equation}
Then the unique solution to problem \eqref{prob-comp} satisfies
\begin{equation}\label{conv-1}
  \lim_{t\to\infty}u(\cdot, t)=0\,,\quad\hbox{uniformly in }\overline\Omega\,.
\end{equation}
\end{theorem}

\begin{proof}
We assume that $\lambda>0$, the other case may similarly be proved.

Let $u\in \mathcal C(\overline\Omega\times [0,\infty))$ be the unique solution to problem \eqref{prob-comp} with $\lambda>0$.
Since $v^1(x,t)=\lambda$ and $v^2(x,t)=0$ define a supersolution and a subsolution, respectively, it follows from the Comparison Principle that
\begin{equation}\label{bas-des}
0\le u(x,t)\le \lambda\,,\quad\hbox{for every } \mbox{ in } \overline\Omega\times (0, +\infty) \,.
\end{equation}

Moreover, fixed $\tau>0$, the function $u^\tau(x,t)=u(x,t+\tau)$ defines a solution with initial datum $u^\tau_0(x)=u(x,\tau)$. Thus, the basic inequality \eqref{bas-des} implies $0\le u^\tau_0(x)\le \lambda$. Appealing again to the Comparison Principle, it yields
$$
0\le u(x,t+\tau)\le u(x,t)\,,\quad\hbox{for every } \mbox{ in } \Omega\hbox{ and for any   }\tau>0\,.
$$
Hence, we obtain that our solution is nonincreasing with respect to  $t$. As a consequence, there exists
$$
w(x)=\lim_{t\to\infty}u(x,t)\,,\quad\hbox{for any  }x\in\overline\Omega\,.
$$
We have to prove that $w(x)=0$ for any  $x\in\overline\Omega$. Observe that this limit function satisfies
$$
w(x)=\int_0^\infty\int_{\Omega_J}J(x-y)\mathcal G(x,u(y;x,t))u(y;x,t)\, dy\, dt+\lambda \,,\quad x\in\overline\Omega
$$
and $w_{\big|_{\Omega_J\backslash \overline\Omega}}\equiv 0$.

Fixed any $x\in\Omega$, consider a sequence $\{t_n\}_{n \in \mathbb{N}}$ satisfying $t_n\to\infty$. We deduce that
$$
\lim_{n\to\infty}u_t(x,t_n)=\int_{\Omega_J}J(x-y)\mathcal G(x,w(y;x))w(y;x)\, dy\,,
$$
and so this limit does not depend on the chosen sequence. Thus, there exists $\lim_{t\to\infty}u_t(x,t)=\ell$ and this limit is nonpositive since our solution is nonincreasing in $t$. (We remark that the limit $\ell$ depends on the considered point $x$.) Assume by contradiction that $\ell<0$. Then there exists $t_0>0$ such that
$$
u_t(x,t)<\frac\ell2\,,\quad\hbox{for any  }t\ge t_0\,.
$$
It follows that $u(x,t)-u(x,t_0)<\frac\ell2(t-t_0)$, which implies
$u(x,t)<\lambda+\frac\ell2(t-t_0)$ and this quantity is negative for $t$ large enough. Since this contradicts \eqref{bas-des}, we have $\ell=0$. Obviously, this argument holds for every $x\in\Omega$, wherewith
$$
\lim_{t\to\infty}u_t(x,t)=\int_{\Omega_J}J(x-y)\mathcal G(x,w(y;x))w(y;x)\, dy=0\,,\quad x \hbox{ in } \Omega\,.
$$
By continuity, we conclude that
\begin{equation}\label{equa1}
  \int_{\Omega_J}J(x-y)\mathcal G(x,w(y;x))w(y;x)\, dy=0\,,\quad x \hbox{ in } \overline\Omega\,.
\end{equation}

Recalling that the function $w$ is the limit of a nonincreasing family of continuous functions, we deduce that $w$ is lower semicontinuous in $\overline\Omega$. So $w$ attains its maximum   in $\overline\Omega$; let $x_0\in \overline\Omega$ satisfy $w(x)\le w(x_0)$ for any  $x\in\overline\Omega$.

Since the function $J$ is radial symmetric, it is positive in an open ball centered at the origin; we denote its radius is $r$
Let $n$ be the integer part of $\dis(x_0,\partial\Omega)/r$.
Applying \eqref{equa1} it yields
$$
\int_{\Omega_J}J(x-y)\mathcal G(x,w(y)-w(x_0))(w(y)-w(x_0))\, dy=0\,.
$$
Since the integrand is nonpositive, it vanishes, so that $w(y)=w(x_0)$ for any  $y\in\overline\Omega$ satisfying $y-x_0\in\supp J$, that is, for any  $y\in\overline\Omega\cap B_1(x_0)$. If $n\ge 1$ and so $B_r(x_0)\subset\Omega$, taking $y_0$ close to the boundary of $B_r(x_0)$ and applying the same argument, we infer that $w(y)=w(x_0)$ for any  $y\in\overline\Omega \cap B_{2r}(x_0)$. We may follow this procedure $n$ times to find some $x\in\Omega$ such that $w(x)=w(x_0)$ and $\dis(x,\partial\Omega)<r$ (this fact can already be attained in the first step if $n=0$). Then
$$
0=\int_{\overline\Omega}J(x-y)\mathcal G(x,w(y)-w(x))(w(y)-w(x))\, dy+\int_{\Omega_J\backslash\overline\Omega}J(x-y)\mathcal G(x,-w(x))(-w(x))\, dy\,.
$$
Notice that both integrands are nonpositive, so that both vanish. We deduce from the first integral that $w$ is constant in $\overline\Omega\cap B_r(x)$ and from the second one that this constant is equal to $0$. Therefore, $w(x_0)=w(x)=0$ and as a consequence $w(x)=0$ for any  $x\in\overline\Omega$.

Recalling that the function $u(x, t)$ is nonincreasing in $t$ and $\displaystyle \lim_{t\to\infty}u(x,t)=0$ for any  $x\in \overline\Omega$, we deduce from Dini's Theorem that this convergence is uniform.
\end{proof}

\medskip

With the help of Theorem \ref{comp1}, we are ready to prove Theorem \ref{comp2}.

\medskip

\begin{proof}[Proof of Theorem \ref{comp2}]
Consider $u^1$ the solution to \eqref{prob-hom} with initial datum $u^1_0(x)=\|u_0\|_{L^{\infty} (\Omega)}$, and $u^2 \equiv 0$. On the one hand, it follows from the Comparison Principle that
$$
0\le u(x,t)\le u^1(x,t)\,,\quad\hbox{for any  }x\in\overline\Omega\hbox{ and }t>0\,.
$$
On the other hand, we deduce from Proposition \ref{comp1} that
$$
\lim_{t\to\infty}u^1(x,t)=0\,,\quad\hbox{uniformly in }\overline\Omega
$$
and thus the result follows.
\end{proof}
\begin{remark}\label{4.5}
As already mentioned, if hypothesis \rife{mu<0} holds true, we have that the decay at $0$ is of exponential type. Indeed,
\begin{align*}
\frac{d}{dt}\int_\Omega u^2(x,t)dx &=2 \ir \ir J(x-y)\mathcal{G}(x,u(y;x,t))\,u(y;x,t)\,u(x,t)\,dy\,dx
\\
& =-\beta \ir \ir J(x-y)(u(y,t)-u(x,t))^2\,dy\,dx.
\end{align*}
Now, due to \cite{Chasse}, there exists a pair $( \lambda_1, \phi  (x)) \in \RR^+ \times  \mathcal{C} (\Omega)$ such that
\[
0<\lambda_1=\inf_{u\in L^2(\Omega)\setminus \{0\}} \frac{\displaystyle\frac{1}{2}\ir\ir J(x-y)(u(y)-u(x))^2dy\,dx}{\displaystyle\int_\Omega u(x)^2dx}
\]
and a function $\phi(x)$ where the infimum is attained. Consequently, we conclude that
\[
\frac{d}{dt}\int_\Omega u^2(x,t)dx\leq -2\beta \lambda_1\int_\Omega u(x,t)^2dx,
\]
and integrating over $[0,t]$, we have  that
$\|u(\cdot,t)\|_{L^2(\Omega)}\leq \|u_0\|_{L^2(\Omega)}\,e^{-\lambda_1  \beta  \ t}\,
$.
\end{remark}

 \section{Proofs of results  about Cauchy Problem}
As in the previous Section, we start by proving the existence and uniqueness result.

\begin{proof}[Proof of Theorem \ref{existence_Cauchy}] For $T>0$ we consider the Banach space
$$
X=\mathcal{C}\left([0,T]; \mathcal{C}(\RR^N)\cap L^\infty(\RR^N)  \right),
$$
endowed with the norm
$$
|||w|||=\max_{0\leq t \leq T} e^{-kMt} \|w(\cdot,t)\|_{{L^{\infty} (\RR^N)}}.
$$
Here $M=2\,\alpha_2$ and $k\geq 1$.

  Now, let $Y$ be the closed ball of $X$ with radius $k\|u_0\|_{L^{\infty} (\RR^N)}$ and centered at the origin. Note that $Y$ is a complete metric space with the induced metric $d(w_1,w_2)=|||w_1-w_2|||$.

   In order to establish the existence and uniqueness of solutions of \eqref{C} via Banach contraction principle, we define the operator $\mathcal{T}:Y\longrightarrow Y$ by
$$
\mathcal{T}(w)(x,t)=\int_0^t\ir\negthickspace\negthickspace\negthickspace J(x-y)w(y;x,\tau)\mathcal{G}(x,w(y;x,\tau))dyd\tau+u_0(x).
$$

Let us first prove that this operator is well defined. Clearly $\mathcal{T}(w)$ is belongs to $X$ and satisfies
\begin{equation}  \label{belongs}
\begin{array}{c}\displaystyle
\|\mathcal{T}(w)(\cdot,t)\|_{{L^{\infty} (\RR^N)}}   \leq \alpha_2 \max_{x\in \RR^N} \int_0^t\ir J(x-y)|w(y;x,s)|dyds +\|u_0\|_{L^{\infty} (\RR^N)}
\\ \displaystyle
 \leq 2 \alpha_2 \int_0^t \|w(\cdot,s)\|_{L^{\infty} (\RR^N)} ds+\|u_0\|_{L^{\infty} (\RR^N)}
  \leq 2 \alpha_2 |||w||| \int_0^t e^{kMs}ds+\|u_0\|_{L^{\infty} (\RR^N)}
\leq  e^{kMt}\|u_0\|_{L^{\infty} (\RR^N)}.
\end{array}
\end{equation}
Therefore,
\begin{align*}
|||\mathcal{T}(w)||| & =\max_{0\leq t \leq T}e^{-kMt}\|\mathcal{T}(w)(\cdot,t)\|_{L^{\infty} (\RR^N)}
  \leq \|u_0\|_{L^{\infty} (\RR^N)}.
\end{align*}
Since $k>1$, we obtain that $|||\mathcal{T}(w)||| \leq k \|u_0\|_{L^\infty(\RR^N)}$ and $\mathcal{T}(w)$ belongs to $Y$.

  Now, let us show that the operator $\mathcal{T}$ is a contraction. By using that $\mathcal{G}$ satisfies $\eqref{Gi}$ and arguing as \eqref{belongs}, we obtain
$$\begin{array}{c}
 \displaystyle
\|\left(\mathcal{T}(w_1)-\mathcal{T}(w_2)\right)(\cdot,t)\|_{{L^{\infty} (\RR^N)}}  \leq \alpha_2 \max_{x\in \RR^N} \int_0^t\ir J(x-y)|w_1(y;x,\tau)-w_2(y;x,\tau)|dyd\tau
\\ \displaystyle
 \leq 2\,\alpha_2 \int_0^t \|w_1(\cdot,\tau)-w_2(\cdot,\tau)\|_{{L^{\infty} (\RR^N)}} d\tau
  \leq 2\,\alpha_2 |||w_1-w_2||| \int_0^te^{kM\tau}d\tau
 \leq \frac{1}{k}\left(e^{kMt}-1\right) |||w_1-w_2|||.
\end{array}
$$
Therefore,
\begin{align*}
d(\mathcal{T}(w_1),\mathcal{T}(w_2)) & \leq \frac{1}{k} |||w_1-w_2||| \displaystyle\max_{0\leq t\leq T} \left(1-e^{-kMt}\right)
 \leq \frac{1}{k}d(w_1,w_2).
\end{align*}
Since $k>1$, $\mathcal{T}$ is a contraction. Hence, using Banach's Fixed Point Theorem there exists $u$ a fix point of  $\mathcal{T}$, that is the unique solution of problem \eqref{C} for $t\in [0,T]$ and belongs to $Y$.
Finally, since $T$ is arbitrary, we  obtain a global solution, $u \in \mathcal{C}\left([0,\infty);\mathcal{C}(\RR^N)\cap L^\infty(\RR^N)\right)$.
\end{proof}

\medskip
Now we can prove the Comparison Principle.
\medskip

\begin{proof}[Proof of Theorem \ref{C-P_Cauchy}]
Set $w=u-v$, then in virtue of \eqref{h} $w$ satisfies
\begin{equation}\label{w}
\left\{
\begin{array}{ll}
 w_t(x,t) = \displaystyle\int_{\RR^N} J(x-y) \,w(y;x,t) \ \psi (x,u(y;x,t),v(y;x,t))dy
&  \mbox{ in } \RR^N \times (0, +\infty) \\[10pt]
w(x,0)\leq 0,
& \mbox{ in } \RR^N,
\end{array}\right.
\end{equation}
where $\psi$ is the function defined in \rife{h}.
Let us consider the following function
$$
\varsigma (x,t)=\left \{\begin{array}{lr}
1 & \textrm{if}\,\, w(x,t)\geq 0,
\\[1.5 ex]
0 & \textrm{if}\,\, w(x,t)< 0.
\end{array}    \right.
$$
Multiplying \eqref{w} by $\varsigma (x,t)$ and taking into account that $w_t(x,t) \varsigma (x,t)=\left(w_+\right)_t (x,t)$ and $w(y,t) \varsigma (x,t)\leq w_+(y,t)$, we obtain,  dropping the positive term $w(x,t) \varsigma (x,t)$, that
\begin{align*}
\left(w_+\right)_t(x,t)&= \ir J(x-y)\,(w(y,t) \varsigma (x,t)-w(x,t) \varsigma (x,t))\,\psi(x,u(y;x,t),v(y;x,t))dy
\\
& \leq \ir J(x-y)\,  w_+  (y,t)\,\psi(x,u(y;x,t),v(y;x,t))dy
 \leq \alpha_2 \ir J(x-y)\, w_+(y,t)dy,
\end{align*}
integrating in $\RR^N$ and by using $\ir J(z)dz=1$, we get
$$
\ir \left(w_+\right)_t(x,t)dx\leq  \alpha_2 \ir w_+ (y,t)dy.
$$
Finally, integrating in $(0,T]$ and since $w_+(x,0)=0$ we can assert, using Fubini's theorem, that
\begin{equation}\label{gronwall}
k(t)\leq \alpha_2 \int_0^t k(\tau)d\tau,
\qquad \mbox{where} \qquad  k(t)=\ir  w_+(x,t)dx.
\end{equation}

Hence, applying Gronwall's Lemma in \eqref{gronwall}, we conclude that $$k(t)\leq 0.$$
Now, since $w_+(x,t)\geq 0$ and by the continuity of $w_+$, we get that $w_+(x,t)=0$ and, consequently, $$u(x,t)\leq v(x,t)$$ for any  $x \in \RR^N,\,\, t>0$.
\end{proof}

Note that the previous proof works locally in time, that is, a supersolution $v$ and a subsolution $u$ defined both for $t\in [0,T]$ verify $u(x,t)\leq v(x,t)$ for any  $x \in \RR^N,\,\, 0\leq t <  T$.

\begin{proof}[Proof of Theorem \ref{principal_Cauchy}]
By Theorem \ref{existence_Cauchy}, for any  $\varepsilon > 0$ there exists  $u^\varepsilon$ the unique solution of problem \eqref{P_e_Cauchy}. Set $w^\varepsilon:=v-u^\varepsilon$, wich satisfies
\begin{equation}\label{w_Cauchy}
\begin{cases}
 w^\varepsilon_t(x,t)=\Delta  v(x,t)+\mu(x) |\nabla   v(x,t)|^2-\mathcal{L}_\varepsilon(u^\varepsilon(x,t)),
& \mbox{ in }  \RR^N\times (0,T],\\[1.5 ex]
 w^\varepsilon(x,0)=0,
& \mbox{ in }    \RR^N,
\end{cases}
\end{equation}
being
$$\mathcal{L}_\varepsilon(u^\varepsilon(x,t))=\displaystyle \frac{C(x)}{\varepsilon^2} \displaystyle \ir \negthickspace \negthickspace J_\varepsilon(x-y)u^\varepsilon(y;x,t)\,\mathcal{G}(x,u^\varepsilon(y;x,t))dy .$$

Now, the proof follows the one of  Theorem \ref{existence}. \\
Choosing $\overline{w}(x,t)=K\varepsilon^\alpha t$ and $\underline{w}(x,t)=-\overline{w}(x,t)$. Then for $K$ sufficiently large we have that $\overline{w}$ and $\underline{w}$ are super and subsolution of \eqref{w_Cauchy} respectively. Therefore, by the principle comparison of Theorem \ref{C-P_Cauchy} we obtain $\underline{w}\leq w^\varepsilon \leq \overline{w}$ and the proof is straightforward.
\end{proof}






As far as the asymptotic behavior is concerned,  we observe that $\hat{J}(\xi)$, the Fourier transform of $J$, satisfies
\begin{equation*}
  \hat{J}(\xi)\leq1- C (J) |\xi|^2+o(|\xi|^2) ,\quad \mbox{as} \,\,  \xi \to 0\,.
\end{equation*}
where the above estimates follows since
\[
\frac{1}{2}\partial^2_{\xi_i\xi_i}\hat{J}(0)=\frac{1}{2}\ir J(z)z_N^2dz=\frac12 C(J)<\infty,
\]
thanks to \rife{(H)}.

\medskip

  For the convenience of the reader we repeat the following Lemma that is proved in  \cite{CM} including also a sketch of the  proof (in order to make this part of the paper self-contained).

\medskip

\begin{lemma}\label{lemma}Let $u \in L^1(\RR^N)\cap L^2(\RR^N)$ and $J$ satisfying hypothesys \eqref{(H)}. In addition, consider
\[
D_J(u)=\ir \left(1-\hat J(\xi) \right)|\hat u(\xi)|^2d\xi.
\]
Then, $\exists \, \tilde{C}=\tilde{C}(N,J)>0$ such that
\[
\| u\|_{L^{2} (\RR^N)}^2 \leq \tilde{C} \max \left\{\|u\|_{L^{12} (\RR^N)}^{\frac{4}{N+2}}\,D_J(u)^{\frac{N}{N+2}},D_J(u)    \right\}\,,
\]
and consequently
\begin{equation}\label{GNS}
\ir \ir J(x-y)\left( u(y)-u(x)\right)^2dxdy \geq K\,\min \left\{\|u\|_{L^{1} (\RR^N)}^{-\frac{4}{N}}\|u\|_{L^{2} (\RR^N)}^{2+\frac{4}{N}},\|u\|_{L^{2} (\RR^N)}^2    \right \}.
\end{equation}

\end{lemma}

\begin{proof}
First, we set the following quantities
\[
C=\max_{|\xi|\geq 1}\,\frac{1}{1-\hat J(\xi)}>0, \hspace{1,5cm} \delta_0=\left(\frac{C\, D_J(u)}{C(N)\|u\|_{L^{1} (\RR^N)}^2 C(J)}   \right)^{\frac{1}{N+2}}\,,
\]
where $C(N)=\frac{N\pi^{N/2}}{2\Gamma\left(\frac{N}{2}+1\right)}$ and $\Gamma$ denotes the Gamma function.
  Since $u \in L^1(\RR^N)\cap L^2(\RR^N)$ it follows that $\hat u \in L^2(\RR^N)$ and consequently we obtain for $0<\delta \leq 1$ that
\begin{align}\label{1.1}
\|\hat u\|_{L^{2} (\RR^N)}^2 = \int_{|\xi|\leq \delta}|\hat u(\xi)|^2d\xi +\int_{|\xi|> \delta}|\hat u(\xi)|^2d\xi
  \leq \|u\|_{L^{1} (\RR^N)}^2 \,\frac{2C(N)}{N}\delta^N+\frac{C}{ C(J) \ \delta^2} D_J (u) \,.
\end{align}


Now, if we assume that $\delta_0\leq 1$. Replacing $\delta$ by $\delta_0$ in \eqref{1.1}, we have
\begin{align}\label{2}
\|\hat u\|_{L^{2} (\RR^N)}^2 \leq C_1 \|u\|_{L^{1} (\RR^N)}^{\frac{4}{N+2}}D_J(u)^{\frac{N}{N+2}}\,,
\end{align}
where $C_1=\left(\frac{2}{N}+1\right)C(N)^{\frac{2}{N+2}}C^{\frac{N}{N+2}}$.
  Alternatively, if we assume that $\delta_0> 1$, i.e.,
\begin{align*}
C(N)\,\|u\|_{L^{1} (\RR^N)}^2<C\,D_J(u),
\end{align*}
choosing $\delta=1$ in \eqref{1.1} and using the above inequality, we get
\begin{align}\label{3}
  \|\hat u\|_{L^{2} (\RR^N)}^2 &\leq \|u\|_{L^{1} (\RR^N)}^2 \,\frac{2C(N)}{N}+C\,D_J(u)
  \leq \left(\frac{2}{N}+1\right)C\, D_J(u).
\end{align}
Finally, using Plancherel's theorem on $\|\hat u\|_{L^{2} (\RR^N)}^2$ and summarizing \eqref{2} and \eqref{3}, it follows that
\[
\| u\|_{L^{2} (\RR^N)}^2 \leq \tilde{C} \max \left\{\|u\|_{L^{1} (\RR^N)}^{\frac{4}{N+2}}D_J(u)^{\frac{N}{N+2}},D_J(u)    \right\}
\]
where $\tilde{C}=\max \left\{C_1, \left(\frac{2}{N}+1\right)C \right\}$ and the proof is concluded.
  Due to the above formula, we can state the following inequality
\[
D_J(u)\geq K\,\min \left\{\|u\|_{L^{1} (\RR^N)}^{-\frac{4}{N}}\|u\|_{L^{2} (\RR^N)}^{2+\frac{4}{N}},\|u\|_{L^{2} (\RR^N)}^2    \right \},
\]
being $K=K(N,J)$.
Thus, it is easy to check that
\[
\ir \ir J(x-y)\left( u(y)-u(x)\right)^2dxdy=-2\ir \left(J\ast u-u\right)\negthickspace(x)\,u(x)\,dx,
\]
having in mind that Fourier transform preserves inner product we deduce \eqref{GNS}
\end{proof}

Next Lemma gives the $L^1$ boundedness  from above or from below of solutions depending on how the function  $\mathcal{G}(x,s)s$ behaves. To be more specific we have the following result.

\begin{lemma}\label{norma-1} Let $u$ be a solution of Cauchy problem \eqref{C} with $0\leq u_0\in L^1(\RR^N)$. Then
\begin{enumerate}
\item[(i)] If $\mathcal{G}$ satisfies \eqref{mu<0}, it follows that $t \mapsto \|u(\cdot,t)\|_{L^{1} (\RR^N)}$ is decreasing on $[0,\infty)$, therefore
 \[\|u(\cdot,t)\|_{L^{1} (\RR^N)}\leq \|u_0\|_{L^{1} (\RR^N)}\,.\]
\item[(ii)] If $\mathcal{G}$ satisfies \eqref{mu>0}, it follows that $t \mapsto \|u(\cdot,t)\|_{L^{1} (\RR^N)}$ is increasing on $[0,\infty)$, therefore
 \[\|u(\cdot,t)\|_{L^{1} (\RR^N)} \geq \|u_0\|_{L^{1} (\RR^N)}\,.\]
\end{enumerate}
\end{lemma}

\begin{proof}
Since $0\leq u_0$ and Comparison Principle of Proposition \ref{C-P_Cauchy} we can assume that $u(x,t)\geq 0$. Furthermore, if $\mathcal{G}(x,s)s\leq \beta s$ for any  $(x,s)\in \RR^N\times\RR$, since
\begin{align*}
\frac{d}{dt}\ir u(x,t)dx & =\ir \ir J(x-y)u(y;x,t) \ \mathcal{G}(x,u(y;x,t))dydx
\\
& \leq \beta \ir \ir J(x-y)(u(y,t)-u(x,t))dydx=0,
\end{align*}
where the last identity follows since, by Fubini Theorem,
$$
\ir \ir J(x-y)(u(y,t)-u(x,t))dydx =  \ir \ir J(x-y)(u(x,t)-u(y,t))dxdy\,.
$$

Hence $\|u(\cdot,t)\|_{L^{1} (\RR^N)}$ is nonincreasing in time and we state $(i)$. Equivalently, if $\mathcal{G}(x,s)s\geq \beta s$ for any  $(x,s)\in \RR^N\times\RR$, reasoning as above we obtain the opposite inequality and, consequently, $\|u(\cdot, t)\|_{L^{1} (\RR^N)}$ is nondecreasing in time and $(ii)$ is proved.
\end{proof}

\medskip

Now we can prove the asymptotic behavior of the solution for $\mathcal{G}$ satisfying \eqref{mu<0},

\begin{theorem} \label{Asymptotic L^q Cauchy m<0}
Let $u$ be a solution of Cauchy problem \eqref{C} with $\mathcal{G}$ satisfying \eqref{mu<0} and positive prescribed data $u_0\in L^1(\RR^N)\cap L^q(\RR^N)$ for $q\geq 2$. Then there exists $C=C(J,N,\beta,q)>0$ such that
\[
\|u(\cdot,t)\|_{L^q (\RR^N)}\leq C\|u_0\|_{L^1 (\RR^N)}t^{-\frac{N}{2}\left(1-\frac{1}{q}\right)},
\]
for any  $t$ sufficiently large.
\end{theorem}
\bigskip

\medskip

\begin{proof}[Proof of Theorem \ref{Asymptotic L^q Cauchy m<0}]

Let $q \geq 2$ and let us multiply the equation in \rife{C} by $u^{q-1} (x,t)$ (observe that $u\geq 0$): thus we have
\begin{align*}
\frac{d}{dt} \  \frac1q \ir u(x,t)^qdx &= \ir u_t(x,t)u(x,t)^{q-1}dx
\\
& \leq  \, \beta\ir \ir J(x-y)(u(y,t)-u(x,t))u(x,t)^{q-1}dxdy
\\
&=- \frac{\beta}{2} \ir \ir J(x-y)(u(y,t)-u(x,t))(u(y,t)^{q-1}-u(x,t)^{q-1})dxdy
\\
& \leq - C(q,\beta) \ir \ir J(x-y)(u(y,t)^{q/2}-u(x,t)^{q/2})^2dxdy,
\end{align*}
where in the last inequality we have used   Lemma \ref{qu}. Hence by \eqref{GNS}, we get
\begin{align*}
\frac{d}{dt}\ir u(x,t)^qdx  \leq  -C \,\min \left\{\|u(\cdot,t)\|_{{L^{\frac{q}2} (\RR^N)}  }^{-\frac{2q}{N}}\|u(\cdot,t)\|_{L^{q} (\RR^N)}^{q\left(1+\frac{2}{N}\right)},\|u(\cdot,t)\|_{L^{q} (\RR^N)}^q    \right \}
\end{align*}
where $C=C(q,\beta,N,J)$. Now, by interpolation $\|u(\cdot,t)\|_{{L^{\frac{q}2} (\RR^N)} }\leq \|u(\cdot,t)\|_{L^{1} (\RR^N)}^{\frac{1}{q-1}}\,\|u(\cdot,t)\|_{L^{q} (\RR^N)}^{\frac{q-2}{q-1}}$ and denoting by $Y(t)=\|u(\cdot,t)\|_{L^{q} (\RR^N)}^q$, we obtain,  in virtue of Lemma \ref{norma-1}, the following differential inequality
\begin{equation}\label{Y(t)}
Y^\prime(t)\leq -C\, \min \left\{\|u_0\|_{L^{1} (\RR^N)}^{-q\gamma}Y(t)^{1+\gamma},Y(t)  \right\}
\end{equation}
being $\gamma=\displaystyle \frac{2}{N(q-1)}$. Therefore, $Y(t)$ is decreasing. We claim that there exists $t_0\geq 0$ such that
\[
Y(t)\leq \|u_0\|_{L^{1} (\RR^N)}^q,\hspace{1,0cm} t\geq t_0.
\]
Indeed, otherwise, using that $Y(t)$ is decreasing, we would have that   $\|u_0\|_{L^{1} (\RR^N)}^q \leq Y(t)$ for any  $t\geq t_0$. Replacing in \eqref{Y(t)} we obtain
$$
Y^\prime(t)\leq -C\,Y(t),\hspace{1,0cm} t\geq t_0,
$$
and integrating on $[t_0, t ]$ we get that $Y(t)\leq Y(t_0)e^{-C(t-t_0)}\to 0$ as $t \to \infty$ which leads to a contradiction and the claim is proved.

  Thus, since
\begin{align*}
Y(t)=Y(t)^{1+\gamma}Y(t)^{-\gamma}
 \geq Y(t)^{1+\gamma}Y(t_0)^{-\gamma}
  \geq Y(t)^{1+\gamma}\|u_0\|_{L^{1} (\RR^N)}^{-q\gamma},
\end{align*}
it follows, by inequality \eqref{Y(t)}, that
\[
Y^\prime(t)\leq -C\, \|u_0\|_{L^{1} (\RR^N)}^{-q\gamma}\,Y(t)^{1+\gamma}, \hspace{1,0cm} t\geq t_0.
\]
Integrating on $[t_0,t]$ we get
\[
Y(t)\leq \frac{\|u_0\|_{L^{1} (\RR^N)}^q}{(\gamma\,C)^{1/\gamma}}\,(t-t_0)^{-1/\gamma}.
\]
Having in mind that $Y(t)=\|u(\cdot,t)\|_{L^{q} (\RR^N)}^q$ and $\displaystyle\frac{-1}{q\,\gamma}=-\frac{N}{2}\left(1-\frac{1}{q} \right)$ we conclude that, for any  time $t$ large enough, $\exists \   C=C(J,N,\beta,q)$, such that
\[
\|u(\cdot,t)\|_{L^{q} (\RR^N)} \leq C \|u_0\|_{L^{1} (\RR^N)}\, t^{-\frac{N}{2}\left(1-\frac{1}{q} \right)}\,.
\]
  \end{proof}

\smallskip

With the help of the above result, we can now prove Theorem \ref{asymp Cauchy}.

\smallskip

\begin{proof}[Proof of Theorem \ref{asymp Cauchy}]
Theorem \ref{Asymptotic L^q Cauchy m<0} covers the case $q\geq 2$, while for  $q\in (1,2]$ the interpolation inequality yields to
\begin{align*}
\|u(\cdot,t)\|_{L^{q} (\RR^N)} & \leq \|u(\cdot,t)\|_{L^{1} (\RR^N)}^{\frac{2}{q}-1}\, \|u(\cdot,t)\|_{L^{2} (\RR^N)}^{2\left(1-\frac{1}{q} \right)}
 \leq C \|u_0\|_{L^{1} (\RR^N)}\,t^{-\frac{N}{2}\left(1-\frac{1}{q} \right)},
\end{align*}
being $C=C(J,N,\beta,q)$ a positive constant.
\end{proof}

\medskip

In order to obtain a decay estimate of the norm of the solution $u$,  for functions $\mathcal{G}_\mu$ with $\mu(x)\geq 0$, a $L^1$ boundedness from  above of $u$ is required. For this purpose, we must to control de $L^\infty$-norm of initial data $u_0$ with respect to function $\mu$.

\begin{lemma}\label{cetraro}
Let $u$ be a solution of of Cauchy problem \eqref{C} with $\mathcal{G}\equiv \mathcal{G}_\mu$, $0\leq \mu \in L^\infty(\RR^N)$ and positive prescribed data $u_0\in L^\infty(\RR^N)\cap \mathcal{C}(\RR^N)$ satisfying $\|u_0\|_{L^\infty(\RR^N)}\|\mu\|_{L^\infty(\RR^N)}=\theta<1$. Then
\begin{equation}\label{energy}
\frac{d}{dt}\|u(\cdot,t)\|_{L^2(\RR^N)}^2\leq -(1-\theta)\ir \ir J(x-y)(u(y,t)-u(x,t))^2dydx.
\end{equation}
If, in addition, $u_0\in L^1(\RR^N)$ then
\begin{equation}\label{u_0 bounded from above}
\|u(\cdot,t)\|_{L^1(\RR^N)}\leq c \|u_0\|_{L^1(\RR^N)},
\end{equation}
with $c=c(\|u_0 \|_{L^{\infty} (\RR^N)}, \| \mu  \|_{L^{\infty} (\RR^N)})>1$.
\end{lemma}

\begin{proof}
Since $u_0\in L^\infty(\RR^N)\cap \mathcal{C}(\RR^N)$, by Theorem \ref{existence_Cauchy} there exists a unique solution of problem \eqref{C} and it satisfies  $u\in \mathcal{C}\left([0,\infty);\mathcal{C}(\RR^N)\cap L^\infty(\RR^N)  \right)$. Moreover, since $0$ and $\|u_0\|_{L^\infty(\RR^N)}$ are sub and supersolution respectively of problem \eqref{C}, we get, due the comparison principle Theorem \ref{C-P_Cauchy}, that
\[
0 \leq u(x,t) \leq \|u_0\|_{L^\infty(\RR^N)}, \hspace{1,2cm} (x,t)\in \RR^N\times [0,\infty).
\]

Let us multiply the equation in \rife{C} by $u(x,t) $ and integrate in $\RR^N$, so that
$$
\frac{d}{dt} \|u (\cdot , t)\|_{L^2 (\RR^N)}^2 = 2  \int_{\RR^N} u_t (x,t) u(x, t) dx
=2 \int_{\RR^N}\int_{\RR^N} J(x-y) u(y;x,t) \ \mathcal{G}_\mu (u(y;x, t))\,u(x,t) \ dy dx
$$
$$
= 2 \int_{\RR^N}\int_{\RR^N} J(x-y)u(y;x,t)\ u(x,t)dydx  + \int_{\RR^N}\int_{\RR^N} J(x-y)\frac{\mu(x)u(y;x,t)^2}{1+\mu^2(x)u(y;x,t)^2}u(x, t) dy  dx
$$
$$
\leq -\ir\ir J(x-y)u(y;x,t)^2dydx+\ir \ir J(x-y)\mu(x)u(y;x,t)^2u(x,t)dydx
$$

$$
=-\ir \ir J(x-y)u(y;x,t)^2(1-\mu(x)u(x,t))dydx
$$

$$
\leq-(1-\theta)\ir \ir J(x-y)(u(y,t)-u(x,t))^2dydx,
$$
which proves the first part of lemma.

In order to get \eqref{u_0 bounded from above}, we compute the derivate of $L^1$-norm of $u$, and  we get
\begin{align*}
\frac{d}{dt}\|u(\cdot,t)\|_{L^1(\RR^N)} & = \ir \ir J(x-y)\mathcal{G}_\mu \big(u(y,t)-u(x,t)\big)\big(u(y,t)-u(x,t)\big)dydx
\\
& =  \ir \ir J(x-y)\frac{\mu(x)}{2}\, \frac{\big(u(y,t)-u(x,t)\big)^2}{1+\mu^2(x)\big(u(y,t)-u(x,t)\big)^2}dxdy
\\
&\leq \frac{\|\mu\|_{L^\infty(\RR^N)}}{2}\ir \ir J(x-y)\big(u(y,t)-u(x,t)\big)^2dydx
\\
& \leq -\frac{\|\mu\|_{L^\infty(\RR^N)}}{2}\, \frac{1}{1-\theta}\, \frac{d}{dt}\|u(\cdot,t)\|_{L^2(\RR^N)}^2,
\end{align*}
where we have used \eqref{energy} in the last inequality. Hence, we obtain  the following differential inequality:
\[
\exists\ c_1 >0 \, : \qquad  \frac{d}{dt}\|u(\cdot,t)\|_{L^1(\RR^N)}+c_1 \frac{d}{dt}\|u(\cdot,t)\|_{L^2(\RR^N)}^2\leq 0 ,
\]
being $c_1=\frac{\|\mu\|_{L^\infty(\RR^N)}}{2(1-\theta)}>0$. Consequently, integrating on $[0,t]$,
\begin{align*}
\|u(\cdot,t)\|_{L^1(\RR^N)} +c_1\|u(\cdot,t)\|_{L^2(\RR^N)}^2 
  \leq \|u_0\|_{L^1(\RR^N)}+c_1 \|u_0\|_{L^\infty(\RR^N)}\, \|u_0\|_{L^1(\RR^N)},
\end{align*}
where we have used the interpolation formula, $\|u_0\|_{L^2(\RR^N)}^2\leq  \|u_0\|_{L^\infty(\RR^N)}\,\|u_0\|_{L^1(\RR^N)}$.
Finally we conclude that $\|u(\cdot,t)\|_{L^1(\RR^N)}\leq c \|u_0\|_{L^1(\RR^N)}$, for $c=1+c_1\|u_0\|_{L^\infty(\RR^N)}$.

\end{proof}

\bigskip

\begin{proof}[Proof of Theorem \ref{L^2 mu>0}]
Applying inequality \eqref{GNS} in \eqref{energy} from Lemma \ref{cetraro}, it follows
\[
\frac{d}{dt} \|u (\cdot , t)\|_{L^2 (\RR^N)}^2\leq -C_1 \min \left\{ \|u(\cdot,t)\|_{L^{1} (\RR^N)}^{-\frac{4}{N}}\|u(\cdot,t)\|_{L^{2} (\RR^N)}^{2+\frac{4}{N}},\|u(\cdot,t)\|_{L^{2} (\RR^N)}^2  \right \},
\]
where $C_1=C_1(\|\mu\|_{L^\infty(\RR^N)},\|u_0\|_{L^\infty(\RR^N)},N,J)>0$. Writing $X(t)=\|u (\cdot , t)\|_{L^2 (\RR^N)}^2$ and using the boundedness of $L^1$-norm in inequality \eqref{u_0 bounded from above} we have that
\[
X^\prime(t)\leq -C_2 \min \left \{\|u_0\|_{L^1(\RR^N)}^{-\frac{4}{N}}\,X(t)^{1+\frac2N},\, X(t)   \right \},
\]
where $C_2=C_2(\|\mu\|_{L^\infty(\RR^N)},\|u_0\|_{L^\infty(\RR^N)},N,J)>0$. Thus, arguing as in proof of Theorem \ref{Asymptotic L^q Cauchy m<0}, we can assume that there exists $t_0\geq0$ such that $X(t)\leq \|u_0\|_{L^1(\RR^N)}^2$ for $t\geq t_0$ and therefore,
\[
X^\prime(t)\leq -C_2\|u_0\|_{L^1(\RR^N)}^{-\frac4N}\,X(t)^{1+\frac2N}, \hspace{1cm} t\geq t_0.
\]
Finally, integrating on $[t_0,t]$, we obtain the $L^2$-norm decay estimate for any $t$ sufficiently large.
\end{proof}

\end{document}